\documentclass[12pt]{amsart}
\usepackage{bbm}
\usepackage{units}
\usepackage{mathtools}
\usepackage{cancel}
\usepackage{amsmath}
\usepackage{amsfonts}
\usepackage{latexsym}
\usepackage{amssymb}
\usepackage{mathrsfs}
\usepackage{amscd}
\usepackage{hyperref}
\usepackage{soul}
\usepackage{psfrag}
\usepackage{graphicx}
\usepackage[normalem]{ulem}
\usepackage{fullpage}
\usepackage[all]{xy}
\usepackage{rotating}
\usepackage{url}
\usepackage{color}
\usepackage{enumerate}
\usepackage{cleveref}
\usepackage[utf8]{inputenc}
\usepackage{bbm}

\usepackage{dsfont}
\usepackage{tikz}
\usetikzlibrary{arrows,backgrounds,arrows.meta,decorations.markings}
\usepackage{tikz-cd}

\numberwithin{equation}{section}
\numberwithin{figure}{section}

\theoremstyle{plain}
\newtheorem{thm}{Theorem}[section]

\theoremstyle{plain}
\newtheorem{lem}[thm]{Lemma}

\theoremstyle{remark}

\theoremstyle{plain}
\newtheorem{cor}[thm]{Corollary}

\theoremstyle{definition}
\newtheorem{defn}[thm]{Definition}

\theoremstyle{definition}

\theoremstyle{definition}

\theoremstyle{plain}
\newtheorem{prop}[thm]{Proposition}

\theoremstyle{plain}

\theoremstyle{definition}

\theoremstyle{plain}

\newtheorem{lemma}[thm]{Lemma}
\newtheorem{proposition}[thm]{Proposition}
\newtheorem{question}[thm]{Question}

\theoremstyle{definition}

\newtheorem{notation}[thm]{Notation}

\newtheorem{example}[thm]{Example}
\newtheorem{remark}[thm]{Remark}
\newtheorem*{ac}{Acknowledgements} 

\DeclareMathOperator{\Tr}{Tr}
\newcommand{\id}{\mathrm{id}} 
\newcommand{\End}{\mathrm{End}}
\newcommand{\ev}{\mathrm{ev}}
\newcommand{\coev}{\mathrm{coev}}
\newcommand{\trace}{\mathrm{tr}}
\newcommand{\VVec}{\mathrm{Vec}}

\newenvironment{restatetheorem}[1]
{
 \par\addvspace{0.25\baselineskip}
 \noindent\textbf{Theorem \ref{#1}.}\itshape
}
{
 \par\addvspace{0.25\baselineskip}
}
\newenvironment{restatecorollary}[1]
{
 \par\addvspace{0.25\baselineskip}
 \noindent\textbf{Corollary \ref{#1}.}\itshape
}
{
 \par\addvspace{0.25\baselineskip}
} 

\newcommand{\comments}[1]{}

\newcommand{\mcal}{\mathcal}

\newcommand{\C}{\mathbb{C}}

\newcommand{\vlon}{\varepsilon}

\newcommand{\ot}{\otimes}

\title{Exchange relations and Frobenius subalgebras}

\author{Mainak Ghosh}
\address{M. Ghosh, Hetao Institute of Mathematics and Interdisciplinary Sciences, Futian District, Shenzhen, China}
\email{ghosh.main@gmail.com}

\author{Sebastien Palcoux}
\address{S. Palcoux, Beijing Institute of Mathematical Sciences and Applications, Huairou District, Beijing, China}
\email{sebastien.palcoux@gmail.com}
\urladdr{https://sites.google.com/view/sebastienpalcoux}

\makeatletter
\def\blfootnote{\gdef\@thefnmark{}\@footnotetext}
\makeatother

\begin{document}
 \global\long\def\vlon{\varepsilon}
 \global\long\def\bt{\bowtie}
 \global\long\def\ul#1{\underline{#1}}
 \global\long\def\ol#1{\overline{#1}}
 \global\long\def\norm#1{\left\|{#1}\right\|}
 \global\long\def\os#1#2{\overset{#1}{#2}}
 \global\long\def\us#1#2{\underset{#1}{#2}}
 \global\long\def\ous#1#2#3{\overset{#1}{\underset{#3}{#2}}}
 \global\long\def\t#1{\text{#1}}
 \global\long\def\lrsuf#1#2#3{\vphantom{#2}_{#1}^{\vphantom{#3}}#2^{#3}}
 \global\long\def\tr{\triangleright}
 \global\long\def\tl{\triangleleft}
 \global\long\def\cc90#1{\begin{sideways}#1\end{sideways}}
 \global\long\def\turnne#1{\begin{turn}{45}{#1}\end{turn}}
 \global\long\def\turnnw#1{\begin{turn}{135}{#1}\end{turn}}
 \global\long\def\turnse#1{\begin{turn}{-45}{#1}\end{turn}}
 \global\long\def\turnsw#1{\begin{turn}{-135}{#1}\end{turn}}
 \global\long\def\fusion#1#2#3{#1 \os{\textstyle{#2}}{\otimes} #3}
 
 \global\long\def\abs#1{\left|{#1}\right |}
 \global\long\def\red#1{\textcolor{red}{#1}}

 \maketitle
 
 \begin{abstract}
 Bisch and Jones established a bijection between the intermediate subfactors of an irreducible subfactor and certain idempotents satisfying exchange relations. 
 In this paper, we generalize this result to Karoubian monoidal categories through Frobenius subalgebras.
 As an application, we show that certain morphisms on a C$^*$-correspondence arise from intermediate C$^*$-subalgebras if and only if they are averaging operators.

 \end{abstract}

 
 
 \section{Introduction}\label{introduction}
 
 A Frobenius algebra in $\VVec$, the category of finite-dimensional vector spaces, is a finite-dimensional unital algebra $A$ that is isomorphic to its dual $A^*$ as an $A$-module. Equivalently, it can be characterized by the existence of a non-degenerate associative bilinear form $\langle \cdot , \cdot \rangle$, as described in \cite[Chapter 4]{SY11}. Results of Quinn and Abrams \cite{Q95,A99} provide the following alternative characterization: a Frobenius algebra is a quintuple $(X,m,e,\delta,\epsilon)$, where $(A,m,e)$ and $(A,\delta,\epsilon)$ are a finite-dimensional algebra and coalgebra, respectively, subject to the condition
 \[
 \left( m \otimes 1_A \right) \circ \left(1_A \otimes \delta\right) = \delta \circ m = \left( 1_A \otimes m \right) \circ \left( \delta \otimes 1_A \right).
 \]
 This second definition allows the notion of Frobenius algebras to be extended to any monoidal category. In the operator algebraic framework, Frobenius algebras appear as the standard invariants of subfactors.
 A subfactor is a unital inclusion of factors. The modern theory of subfactors was initiated with V. Jones' landmark results in \cite{J83}. The standard invariant of a finite-index subfactor of a $\rm{II}_1$ factor was first defined as a $\lambda$-lattice in \cite{P95}. In \cite{L94, M03}, a Q-system—which is a unitary version of a separable Frobenius algebra object in a C*-tensor category or a C*-2-category—was introduced as an alternative axiomatization of the standard invariant of a finite-index subfactor.
 
 The study of intermediate subfactors of an irreducible subfactor was initiated by D. Bisch \cite{B94}, where the notions of \textit{biprojections} and \textit{exchange relations} first appeared. The notion of exchange relations has been extensively studied in subfactor theory \cite{B94, BJ00, L02, L16}. In \cite{BJ00}, they were defined for singly generated subfactor planar algebras, and later Z. Landau \cite{L02} extended them to subfactor planar algebras with multiple generators. Recently, exchange relations were generalized to monoidal categories through Frobenius algebras \cite[Theorem 3.19]{GP25}. 
 Note that a unital selfdual idempotent satisfying the exchange relations is not necessarily a biprojection (see \S\ref{sub:ERVec}).
 Biprojections on $M_n(\C)$ admit an elementary characterization (Lemma~\ref{lem:12345}). We describe all Frobenius subalgebras of $M_2(\C)$ (Proposition~\ref{prop:matrix_Frob_subalg}) and classify all the biprojections on it (Proposition~\ref{prop:bipro3}).
 
 A Banach algebra endomorphism satisfying the exchange relations is called an \emph{averaging operator} in \cite{M66} (Definition~\ref{def:average}). 
 Spectral properties of averaging operators have been studied extensively in~\cite{GM68}. 
 
 Bisch and Jones proved in \cite[Proposition 3.6]{BJ00} that there is a one-to-one correspondence between intermediate subfactors of an irreducible subfactor and unital projections satisfying exchange relations. The categorical analogues of an irreducible subfactor and its intermediate subfactors are a connected Frobenius algebra and its Frobenius subalgebras, respectively (see \cite{GP25}). Our main goal in this paper is to generalize the result of Bisch and Jones, using their foundational ideas, to Karoubian monoidal categories under a weaker assumption. The connectedness (irreducibility) assumption can be relaxed to unitality of the morphism. The main result is as follows:
 \begin{restatetheorem}{thm:main} 
 Let $X$ be a Frobenius algebra in a Karoubian monoidal category $\mcal C$. Then a unital selfdual idempotent $b: X \to X$ arises from a Frobenius subalgebra if and only if it satisfies the exchange relation.
 \end{restatetheorem}
 %
 \begin{restatecorollary}{connectedsubalgcor}
 Let $X$ be a connected Frobenius algebra in a Karoubian monoidal category $\mcal C$. Then a selfdual idempotent $b: X \to X$ arises from a Frobenius subalgebra if and only if it satisfies the exchange relation.
 \end{restatecorollary}
 %
 In the connected case, if a unital selfdual idempotent satisfies the exchange relations, then it is a biprojection (Definition \ref{def:bipro}), but the converse is unknown. 
 
 A normal biprojection (Definition \ref{def:normalbipro}), satisfies the exchange relations (Proposition \ref{normalERprop}), so arises from a Frobenius subalgebra (Corollary \ref{cor:normalbi}). 
 
 %
We now state a connected unitary analogue of \Cref{thm:main} (see Remark~\ref{rk:DualVsAdjoint}), which reformulates the above result of Bisch and Jones in the setting of unitary tensor categories.
 \begin{restatetheorem}{thm:uniMain}
 Let $X$ be a connected unitary Frobenius algebra in a unitary tensor category $\mcal C$. A unital self-adjoint idempotent $b: X \to X$ arises from a unitary Frobenius subalgebra if and only if it satisfies the exchange relations.
 \end{restatetheorem}
 %
 
 The study of exchange relations for inclusions of unital C*-algebras has not been previously undertaken. In this paper, we initiate this study. As an application of \Cref{thm:uniMain}, we can characterize averaging operators on C*-correspondences in terms of intermediate C*-subalgebras of irreducible finite-index unital inclusions of C*-algebras:
 
 \begin{restatecorollary}{cor:C*algthm}
 Let $(A \subseteq B, E)$ be an irreducible finite-index unital inclusion of C*-algebras. 
 Then an $A$-bimodule unital self-adjoint morphism $T : {_A}B_A \to {_A}B_A$ arises from an intermediate C*-subalgebra if and only if $T$ is an averaging operator.
 \end{restatecorollary}
 
 The structure of the paper is as follows. In \S\ref{prelim}, we recall the necessary background, including basic results, definitions, and pictorial notations. In \S\ref{biprojER}, we study biprojections and exchange relations for Frobenius algebras. Finally, in \S\ref{main}, we establish the main results.

 \section{Preliminaries}\label{prelim}
 In this section, we recall some basic definitions and highlight key results that will be used later. Let $\mcal C$ be a monoidal category with unit object $\mathbbm{1}$. We employ the graphical calculus for monoidal categories, interpreting diagrams from top to bottom.
 
 
 \begin{defn}
 A \textit{unital algebra} in $\mcal C$ is a triple $(X,m,e)$, where $X$ is an object of $\mcal C$, $m \in \text{Hom}_{\mcal C}(X \otimes X, X)$ is the multiplication morphism, and $e \in \text{Hom}_{\mcal C}(\mathbbm{1}, X)$ is the unit morphism, depicted as follows:
 \vspace*{-5mm} 
 \[m = \raisebox{-6mm}{
 \begin{tikzpicture}
 \draw[blue,in=-90,out=-90,looseness=2] (-0.5,0.5) to (-1.5,0.5);
 \draw[blue] (-1,-.1) to (-1,-.6);
 \node[left,scale=0.7] at (-1,-.4) {$X$};
 \node[left,scale=0.7] at (-1.6,0.5) {$X$};
 \node[right,scale=0.7] at (-.5,.5) {$X$};
 \end{tikzpicture}} \ \ \ \ \ \ \ \ \ \
 e= \raisebox{-6mm}{
 \begin{tikzpicture}
 \draw [blue] (-0.8,-.6) to (-.8,.6);
 \node at (-.8,.6) {${\color{blue}\bullet}$};
 \node[left,scale=.8] at (-.8,.5) {$\mathbbm{1}$};
 \node[left,scale=0.7] at (-.8,-.5) {$X$};
 \end{tikzpicture}}
 \vspace*{-2mm}
 \]
 These satisfy the following axioms:
 \begin{itemize}
 \item (Associativity) $m \circ (m \ot \text{id}_X) = m \circ (\text{id}_{X} \ot m)$,
 
 \item (Unitality) $m \circ (e \ot \text{id}_X) = \text{id}_X = m \circ (\text{id}_X \circ e)$,
 \end{itemize}
 depicted as follows:
 \vspace*{-3mm}
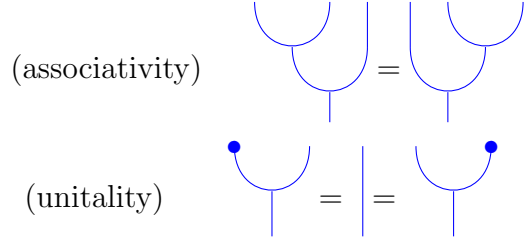
 \[\text{(associativity)} \ \ \ \ \raisebox{-6mm}{
 \begin{tikzpicture}[rotate=180,transform shape]
 \draw[blue,in=90,out=90,looseness=2] (0,0) to (1,0);
 \draw[blue,in=90,out=90,looseness=2] (0.5,.6) to (-.5,.6);
 \draw[blue] (-.5,.6) to (-.5,0);
 \draw[blue] (0,1.2) to (0,1.6);
 \end{tikzpicture}}
 = \raisebox{-6mm}{\begin{tikzpicture}[rotate=180,transform shape]
 \draw[blue,in=90,out=90,looseness=2] (0,0) to (1,0);
 \draw[blue,in=90,out=90,looseness=2] (.5,.6) to (1.5,.6);
 \draw[blue] (1.5,.6) to (1.5,0);
 \draw[blue] (1,1.2) to (1,1.6);
 \end{tikzpicture}} \]
 \vspace*{-2mm}
 \[ \text{(unitality)} \ \ \ \ 
 \raisebox{-4mm}{
 \begin{tikzpicture}[rotate=180,transform shape]
 \draw[blue,in=90,out=90,looseness=2] (0,0) to (1,0);
 \node at (1,0) {$\textcolor{blue}{\bullet}$};
 \draw[blue] (.5,.6) to (.5,1.2);
 \end{tikzpicture}}
 =
 \raisebox{-4mm}{
 \begin{tikzpicture}
 \draw[blue] (0,0) to (0,1.2);
 \end{tikzpicture}} 
 = 
 \raisebox{-4mm}{
 \begin{tikzpicture}[rotate=180,transform shape]
 \draw[blue,in=90,out=90,looseness=2] (0,0) to (1,0);
 \node at (0,0) {$\textcolor{blue}{\bullet}$};
 \draw[blue] (.5,.6) to (.5,1.2);
 \end{tikzpicture}} \]
 \end{defn}

 \begin{defn}
 A \textit{counital coalgebra} in $\mcal C$ is a triple $(X, \delta, \epsilon)$, where $X$ is an object of $\mcal C$, $\delta \in \text{Hom}_{\mcal C}(X, X \otimes X)$ is the comultiplication morphism, and $\epsilon \in \text{Hom}_{\mcal C}(X, \mathbbm{1})$ is the counit morphism, depicted as follows:
 \vspace*{-4mm}
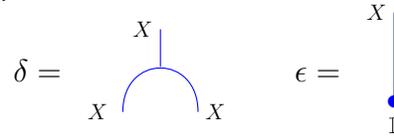
 \[ \delta = \raisebox{-6mm}{
 \begin{tikzpicture}
 \draw[blue,in=90,out=90,looseness=2] (-0.5,0.5) to (-1.5,0.5);
 \draw[blue] (-1,1.1) to (-1,1.6);
 \node[left,scale=0.7] at (-1,1.6) {$X$};
 \node[left,scale=0.7] at (-1.6,0.5) {$X$};
 \node[right,scale=0.7] at (-.5,.5) {$X$};
 \end{tikzpicture}} \ \ \ \ \ \ 
 \epsilon = \raisebox{-8mm}{
 \begin{tikzpicture}
 \draw [blue] (-0.8,-.6) to (-.8,.6);
 \node at (-.8,-.6) {${\color{blue}\bullet}$};
 \node[left,scale=0.7] at (-.8,.6) {$X$};
 \node[scale=.8] at (-.8,-.9) {$\mathbbm{1}$};
 \end{tikzpicture}} 
 \vspace*{-2mm}\]
 These satisfy the following axioms :
 \begin{itemize}
 \item (Coassociativity) $(\delta \ot \text{id}_X) \circ \delta = (\text{id}_X \ot \delta) \circ \delta $,
 
 \item (Counitality) $(\epsilon \ot \text{id}_X) \circ \delta = \text{id}_X = (\text{id}_X \ot \epsilon) \circ \delta$,
 \end{itemize}
 depicted as follows:
 \vspace*{-2mm}
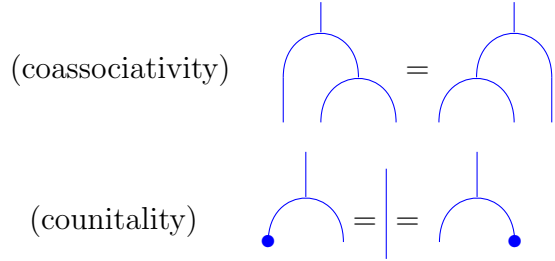
 \[\text{(coassociativity)} \ \ \ \ \raisebox{-6mm}{
 \begin{tikzpicture}
 \draw[blue,in=90,out=90,looseness=2] (0,0) to (1,0);
 \draw[blue,in=90,out=90,looseness=2] (0.5,.6) to (-.5,.6);
 \draw[blue] (-.5,.6) to (-.5,0);
 \draw[blue] (0,1.2) to (0,1.6);
 \end{tikzpicture}}
 = \raisebox{-6mm}{\begin{tikzpicture}
 \draw[blue,in=90,out=90,looseness=2] (0,0) to (1,0);
 \draw[blue,in=90,out=90,looseness=2] (.5,.6) to (1.5,.6);
 \draw[blue] (1.5,.6) to (1.5,0);
 \draw[blue] (1,1.2) to (1,1.6);
 \end{tikzpicture}} \]
 
 \[ \text{(counitality)} \ \ \ \ 
 \raisebox{-4mm}{
 \begin{tikzpicture}
 \draw[blue,in=90,out=90,looseness=2] (0,0) to (1,0);
 \node at (0,0) {$\textcolor{blue}{\bullet}$};
 \draw[blue] (.5,.6) to (.5,1.2);
 \end{tikzpicture}} = \raisebox{-4mm}{\begin{tikzpicture}
 \draw[blue] (0,0) to (0,1.2);
 \end{tikzpicture}} = \raisebox{-4mm}{
 \begin{tikzpicture}
 \draw[blue,in=90,out=90,looseness=2] (0,0) to (1,0);
 \node at (1,0) {$\textcolor{blue}{\bullet}$};
 \draw[blue] (.5,.6) to (.5,1.2);
 \end{tikzpicture}} \]
 \end{defn}
 
 \begin{defn}\label{frobalgdefn}
 A \emph{Frobenius algebra} in $\mcal C$ is a quintuple $(X, m, e, \delta, \epsilon)$, where $(X, m, e)$ is a unital algebra and $(X, \delta, \epsilon)$ is a counital coalgebra, satisfying the following axiom:
 \[\text{(Frobenius)} \ \ (\text{id}_X \ot m) \circ (\delta \ot \text{id}_X) = \delta \circ m = (m \ot \text{id}_X) \circ (\text{id}_X \ot \delta), \]
 depicted as follows:
 \vspace*{-5mm}
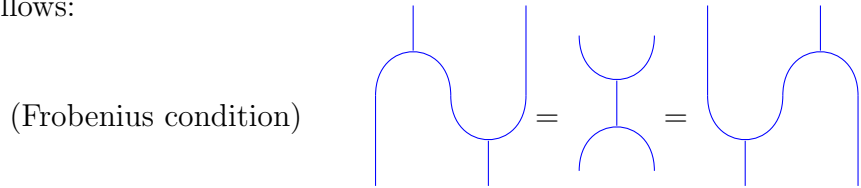
 \[ \text{(Frobenius condition)} \ \ \ \ \ \ \raisebox{-8mm}{
 \begin{tikzpicture}
 \draw[blue,in=90,out=90,looseness=2] (0,0) to (1,0);
 \draw[blue,in=-90,out=-90,looseness=2] (1,0) to (2,0);
 \draw[blue] (.5,.6) to (.5,1.2);
 \draw[blue] (1.5,-.6) to (1.5,-1.2);
 \draw[blue] (0,0) to (0,-1.2);
 \draw[blue] (2,0) to (2,1.2);
 \end{tikzpicture}} =
 \raisebox{-6mm}{
 \begin{tikzpicture}
 \draw[blue,in=90,out=90,looseness=2] (0,0) to (1,0);
 \draw[blue] (.5,.6) to (.5,1.2);
 \draw[blue,in=-90,out=-90,looseness=2] (0,1.8) to (1,1.8);
 \end{tikzpicture}} =
 \raisebox{-8mm}{
 \begin{tikzpicture}
 \draw[blue,in=-90,out=-90,looseness=2] (0,0) to (1,0);
 \draw[blue,in=90,out=90,looseness=2] (1,0) to (2,0);
 \draw[blue] (.5,-.6) to (.5,-1.2);
 \draw[blue] (1.5,.6) to (1.5,1.2);
 \draw[blue] (0,0) to (0,1.2);
 \draw[blue] (2,0) to (2,-1.2);
 \end{tikzpicture}} \]
 \end{defn}
 
 \begin{remark}\label{dualrem}
 Let $(X, m, e, \delta, \epsilon)$ be a Frobenius algebra. Then, as shown for example in \cite[Lemma~2.5]{GP25}, $X$ is selfdual (i.e., $X^* = X$), with evaluation morphism 
 \(
 \text{ev}_X \coloneqq \epsilon \circ m
 \) 
 and coevaluation morphism 
 \(
 \text{coev}_X \coloneqq \delta \circ e,
 \)
 as illustrated below:
 \vspace*{-2mm}
 \[ \raisebox{-4mm}{\begin{tikzpicture}
 \draw[blue,in=-90,out=-90,looseness=2] (-0.5,0.5) to (-1.6,0.5);
 \end{tikzpicture}}
 \coloneqq \,
 \raisebox{-6mm}{\begin{tikzpicture}
 \draw[blue,in=-90,out=-90,looseness=2] (-0.5,0.5) to (-1.5,0.5);
 \node at (-1,-.6) {${\color{blue}\bullet}$};
 \draw[blue] (-1,-.1) to (-1,-.6);
 \end{tikzpicture}} \hspace*{6mm} \text{and} \hspace*{6mm} \raisebox{-2mm}{\begin{tikzpicture}
 \draw[blue,in=90,out=90,looseness=2] (-0.5,0.5) to (-1.6,0.5);
 \end{tikzpicture}} \coloneqq \, \raisebox{-4mm}
 {\begin{tikzpicture}
 \draw[blue,in=90,out=90,looseness=2] (-0.5,0.5) to (-1.5,0.5);
 \node at (-1,1.6) {${\color{blue}\bullet}$};
 \draw[blue] (-1,1.1) to (-1,1.6);
 \end{tikzpicture}} \]
 \end{remark}
 
 \begin{defn}
 Let $(X, m, e)$ and $(X', m', e')$ be unital algebras in $\mcal C$. A morphism $f : X \to X'$ is called a \textit{unital algebra morphism} if it satisfies
 \[
 e' = f \circ e \quad \text{and} \quad m' \circ (f \otimes f) = f \circ m.
 \]
 This condition is illustrated as follows:
 
 \vspace*{-4mm}
 \[ \raisebox{-6mm}{\begin{tikzpicture}
 \draw[blue] (0,0) to (0,1);
 \node at (0,1) {$\textcolor{blue}{\bullet}$};
 \node[left,scale=0.7] at (0,-.2) {$X'$};
 \end{tikzpicture}}
 =
 \raisebox{-8mm}{\begin{tikzpicture}
 \draw[blue] (0,0) to (0,1.5);
 \node[draw,thick,rounded corners, fill=white,minimum width=15] at (0,.7){$f$};
 \node at (0,1.5) {$\textcolor{blue}{\bullet}$};
 \node[left,scale=0.7] at (0,0) {$X'$};
 \end{tikzpicture}}
 \hspace*{12mm}
 \raisebox{-8mm}{
 \begin{tikzpicture}
 \draw[blue,in=-90,out=-90,looseness=2] (-0.5,0.5) to (-1.5,0.5);
 \draw[blue] (-1,-.1) to (-1,-.6);
 \draw[blue] (-.5,.6) to (-.5,1.2);
 \draw[blue] (-1.5,.6) to (-1.5,1.2);
 \node[draw,thick,rounded corners, fill=white,minimum width=15] at (-.5,.6){$f$};
 \node[draw,thick,rounded corners, fill=white,minimum width=15] at (-1.5,.6){$f$};
 \node[left,scale=0.7] at (-1,-.6) {$X'$};
 \node[left,scale=0.7] at (-1.5,1.2) {$X$};
 \node[right,scale=0.7] at (-.5,1.2) {$X$};
 \end{tikzpicture}} 
 =
 \raisebox{-7mm}{
 \begin{tikzpicture}
 \draw[blue,in=-90,out=-90,looseness=2] (-0.5,0.5) to (-1.5,0.5);
 \draw[blue] (-1,-.1) to (-1,-1);
 \node[draw,thick,rounded corners, fill=white,minimum width=15] at (-1,-.5){$f$};
 \node[left,scale=0.7] at (-1,-1.2) {$X'$};
 \node[left,scale=0.7] at (-1.6,0.5) {$X$};
 \node[right,scale=0.7] at (-.5,.5) {$X$};
 \end{tikzpicture}}
 \]
 \end{defn}
 
Let $f : X \to Y$ be a morphism in a monoidal category $\mathcal{C}$, where $X$ and $Y$ admit left duals $X^*$ and $Y^*$. The left dual $f^* : Y^* \to X^*$ is defined pictorially as follows:
 \[f^* \coloneqq \raisebox{-8mm}{
 \begin{tikzpicture}
 \draw [blue] (0,0) to (0,1);
 \draw[blue,in=90,out=90,looseness=2] (0,1) to (.6,1);
 \draw[blue,in=-90,out=-90,looseness=2] (0,0) to (-.6,0);
 \draw[blue] (.6,1) to (.6,.1);
 \draw[blue] (-.6,0) to (-.6,.9);
 \node[draw,thick,rounded corners,fill = white] at (0,.4) {$f$};
 \node[scale = .8] at (-.6, 1) {$Y^*$};
 \node[scale = .8] at (.6,-.1) {$X^*$};
 \end{tikzpicture}}.\]
 
 \begin{defn}[\cite{GP25}]\label{Frobsubalgdefn} 
 A Frobenius algebra $X'$ is called a \textit{Frobenius subalgebra} of a Frobenius algebra $X$ in $\mcal C$ if there exists a unital algebra monomorphism $i : X' \to X$ such that $i^{**} = i$ and $i^* \circ i = \text{id}_{X'}$ (equivalently, $i$ is counital). We then define the selfdual unital idempotent
 \(
 b_{X'} \coloneqq i \circ i^* \in \End_{\mcal C}(X).
 \)
 \end{defn}

\begin{notation} \label{rk:dagger}
Let $\mathcal{C}$ be a unitary tensor category. We denote its unitary structure, also called the adjoint, by $\dagger: \mathcal{C} \to \mathcal{C}$. In particular, it is an involutive contravariant functor that fixes every object and satisfies $(a \circ b)^\dagger = b^\dagger \circ a^\dagger$ for all morphisms $a,b$ in $\mathcal{C}$.
\end{notation}

 \section{Biprojections and exchange relations}\label{biprojER}
 The notions of biprojection and exchange relations were first introduced in \cite{B94}, and later in \cite{BJ00} and \cite{L02}, to characterize intermediate subfactors of an irreducible subfactor. This result is generalized in \S\ref{main}. 
 As background, this section extends these notions to linear monoidal categories (without assuming unitarity), examines them for Frobenius algebras in $\VVec$, and provides a complete classification of biprojections on $M_2(\C)$.
 
 \subsection{Biprojections and exchange relations in monoidal categories}\
 \begin{defn} \label{def:bipro}
 Let $X$ be a Frobenius algebra in a linear monoidal category $\mcal C$. A selfdual unital idempotent $b : X \to X$ is called a \emph{biprojection} if it is \emph{convolution-stable}, i.e., there exists a nonzero scalar $\lambda$ such that
 \vspace*{-3mm}
 \[ b*b \coloneqq m \circ (b \otimes b) \circ \delta = \raisebox{-9mm}{
 \begin{tikzpicture}
 \draw[blue,in=90,out=90,looseness=2] (0,0) to (1,0);
 \draw[blue,in=-90,out=-90,looseness=2] (0,0) to (1,0);
 \draw[blue] (.5,-.6) to (.5,-1);
 \draw[blue] (.5,.6) to (.5,1);
 \node[draw,thick,rounded corners, fill=white,minimum width=20] at (0,0){$b$};
 \node[draw,thick,rounded corners, fill=white,minimum width=20] at (1,0){$b$}; 
 \end{tikzpicture}} = \lambda \, b \]
 In the unitary setting, a \emph{unitary biprojection} $b$ is a self-adjoint biprojection (i.e. $b^\dagger = b$).
 \end{defn}
 By \cite[Proposition~3.6]{BJ00} and \cite[Theorem~2 and Corollary~2.1]{L02}, in the connected case, the notion of a unitary biprojection coincides with the usual notion of a biprojection in the $2$-box space of an irreducible subfactor planar algebra.

 \begin{defn}\label{ERdefn} 
 Let $X$ be a Frobenius algebra in a monoidal category $\mcal C$. A morphism $b \in \End_{\mcal C}(X)$ is said to satisfy the \textit{exchange relations} if:
 \vspace*{-1mm}
 \[ \raisebox{-14mm}{
 \begin{tikzpicture}
 \draw[blue,in=-90,out=-90,looseness=2] (-0.5,0.5) to (-1.5,0.5);
 \draw[blue] (-1,-.1) to (-1,-1);
 \draw[blue] (-1.5,.5) to (-1.5,1.2);
 \draw[blue] (-.5,.5) to (-.5,1.2);
 \node[draw,thick,rounded corners, fill=white] at (-1.5,.6) {$b$};
 \node[draw,thick,rounded corners, fill=white] at (-1,-.5) {$b$};
 \end{tikzpicture}} 
 =
 \raisebox{-14mm}{
 \begin{tikzpicture}
 \draw[blue,in=-90,out=-90,looseness=2] (-0.5,0.5) to (-1.5,0.5);
 \draw[blue] (-1,-.1) to (-1,-1);
 \draw[blue] (-1.5,.6) to (-1.5,1.2);
 \draw[blue] (-.5,.6) to (-.5,1.2);
 \node[draw,thick,rounded corners, fill=white] at (-1.5,.6) {$b$};
 \node[draw,thick,rounded corners, fill=white] at (-.5,.6) {$b$};
 \end{tikzpicture}}
 =
 \raisebox{-14mm}{
 \begin{tikzpicture}
 \draw[blue,in=-90,out=-90,looseness=2] (-0.5,0.5) to (-1.5,0.5);
 \draw[blue] (-1,-.1) to (-1,-1);
 \draw[blue] (-1.5,.5) to (-1.5,1.2);
 \draw[blue] (-.5,.5) to (-.5,1.2);
 \node[draw,thick,rounded corners, fill=white] at (-1,-.5) {$b$};
 \node[draw,thick,rounded corners, fill=white] at (-.5,.6) {$b$};
 \end{tikzpicture}}\]
 \end{defn}
 
 \begin{remark}\label{ERdualrem}
 By~\cite[Lemma~2.10]{GP25}, if $b$ is selfdual, the first relation is equivalent to the second, and both are equivalent to:
 \vspace*{-3mm}
 \[ \raisebox{-6mm}{
 \begin{tikzpicture}
 \draw[blue,in=90,out=90,looseness=2] (-0.5,0.5) to (-1.5,0.5);
 \draw[blue] (-1,1.1) to (-1,2);
 \draw[blue] (-.5,.5) to (-.5,-.2);
 \draw[blue] (-1.5,.5) to (-1.5,-.2);
 \node[draw,thick,rounded corners, fill=white,minimum width=15] at (-1.5,.4){$b$};
 \node[draw,thick,rounded corners, fill=white,minimum width=15] at (-1,1.55){$b$};
 \end{tikzpicture}}
 =
 \raisebox{-6mm}{
 \begin{tikzpicture}
 \draw[blue,in=90,out=90,looseness=2] (-0.5,0.5) to (-1.5,0.5);
 \draw[blue] (-1,1.1) to (-1,2);
 \draw[blue] (-.5,.5) to (-.5,-.2);
 \draw[blue] (-1.5,.5) to (-1.5,-.2);
 \node[draw,thick,rounded corners, fill=white,minimum width=15] at (-1.5,.4){$b$};
 \node[draw,thick,rounded corners, fill=white,minimum width=15] at (-.5,.4){$b$};
 \end{tikzpicture}}
 =
 \raisebox{-6mm}{
 \begin{tikzpicture}
 \draw[blue,in=90,out=90,looseness=2] (-0.5,0.5) to (-1.5,0.5);
 \draw[blue] (-1,1.1) to (-1,2);
 \draw[blue] (-.5,.5) to (-.5,-.2);
 \draw[blue] (-1.5,.5) to (-1.5,-.2);
 \node[draw,thick,rounded corners, fill=white,minimum width=15] at (-.5,.4){$b$};
 \node[draw,thick,rounded corners, fill=white,minimum width=15] at (-1,1.55){$b$};
 \end{tikzpicture}}\]
 \end{remark}
 
 %
 %
 
 
 \begin{defn} \label{def:fourier} 
 Let $X$ be a Frobenius algebra in a monoidal category $\mcal C$. The \textit{Fourier transform} 
 \(
 \mcal F : \End_{\mcal C}(X) \to \End_{\mcal C}(X \otimes X)
 \) 
 is defined by
 \(
 \mcal F(a) = (\text{id}_X \otimes m) \circ (\text{id}_X \otimes a \otimes \text{id}_X) \circ (\delta \otimes \text{id}_X),
 \) 
 with left inverse
 \(
 \mcal F^{-1}(x) = (\epsilon \otimes \text{id}_X) \circ x \circ (\text{id}_X \otimes e).
 \) 
 These transformations are illustrated below:
 \vspace*{-3mm}
 \[ \mathcal F (a) = \raisebox{-8mm}{
 \begin{tikzpicture}
 \draw[blue,in=90,out=90,looseness=2] (0,0) to (1,0);
 \draw[blue,in=-90,out=-90,looseness=2] (1,0) to (2,0);
 \draw[blue] (.5,.6) to (.5,1);
 \draw[blue] (1.5,-.6) to (1.5,-1);
 \draw[blue] (0,0) to (0,-1);
 \draw[blue] (2,0) to (2,1);
 \node[draw,thick,rounded corners, fill=white,minimum width=20] at (1,0) {$a$};
 \end{tikzpicture}}
 \hspace*{8mm} \ \ \t{and} \ \ \ \ \ \
 \mathcal F^{-1} (x) = \raisebox{-6mm}{\begin{tikzpicture}
 \draw[blue] (0,0) to (0,1);
 \draw[blue] (.5,0) to (.5,1);
 \node[draw,thick,rounded corners, fill=white,minimum width=22] at (.25,.5) {$x$};
 \node at (0,0) {$\textcolor{blue}{\bullet}$};
 \node at (.5,1) {$\textcolor{blue}{\bullet}$};
 \end{tikzpicture}}
 \] 
 \end{defn}
 
 \begin{defn}[\cite{T98,P18}] \label{def:normalbipro}
 A morphism $a \in \End_{\mcal C}(X)$ is called \textit{normal} if both $a$ and $\mcal F(a)$ are central in their respective endomorphism spaces.
 \end{defn}
 
 The notion of a normal biprojection generalizes that of a normal subgroup. It was introduced in \cite{P18}, building on \cite{T98}, to define cyclic subfactor planar algebras. By Proposition~\ref{normalERprop}, they satisfy the exchange relations. Without the normality assumption, see Question~\ref{qu:BiER}.

 \begin{prop}\label{normalERprop}
 A selfdual morphism $b:X \to X$ whose Fourier transform is central satisfies the exchange relations.
 \end{prop}
 
 \begin{proof}
 Since $\mathcal{F}(b)$ is central in $\End_{\mathcal{C}}(X \otimes X)$, it commutes with $\id_X \otimes b$. Therefore,
 \vspace*{-1mm}
 \[\raisebox{-8mm}{
 \begin{tikzpicture}
 \draw[blue,in=90,out=90,looseness=2] (0,0) to (1,0);
 \draw[blue,in=-90,out=-90,looseness=2] (1,0) to (2,0);
 \draw[blue] (.5,.6) to (.5,1);
 \draw[blue] (1.5,-.6) to (1.5,-1.2);
 \draw[blue] (0,0) to (0,-1.2);
 \draw[blue] (2,0) to (2,1);
 \node[draw,thick,rounded corners, fill=white,scale=.8] at (1,0) {$b$};
 \node[draw,thick,rounded corners, fill=white,scale=.8] at (2,0) {$b$};
 \end{tikzpicture}}
 =
 \raisebox{-8mm}{
 \begin{tikzpicture}
 \draw[blue,in=90,out=90,looseness=2] (0,0) to (1,0);
 \draw[blue,in=-90,out=-90,looseness=2] (1,0) to (2,0);
 \draw[blue] (.5,.6) to (.5,1);
 \draw[blue] (1.5,-.6) to (1.5,-1.2);
 \draw[blue] (0,0) to (0,-1.2);
 \draw[blue] (2,0) to (2,1);
 \node[draw,thick,rounded corners, fill=white,scale=.8] at (1,0) {$b$};
 \node[draw,thick,rounded corners, fill=white,scale=.8] at (1.5,-.9) {$b$};
 \end{tikzpicture}}\]
 Composing with $\epsilon \otimes \id_X$ and using the counitality of $X$ yields
 \vspace*{-1mm}
 \[
 \raisebox{-14mm}{
 \begin{tikzpicture}
 \draw[blue,in=-90,out=-90,looseness=2] (-0.5,0.5) to (-1.5,0.5);
 \draw[blue] (-1,-.1) to (-1,-.9);
 \draw[blue] (-1.5,.6) to (-1.5,1);
 \draw[blue] (-.5,.6) to (-.5,1);
 \node[draw,thick,rounded corners, fill=white,scale=.8] at (-1.5,.6) {$b$};
 \node[draw,thick,rounded corners, fill=white,scale=.8] at (-.5,.6) {$b$};
 \end{tikzpicture}}
 =
 \raisebox{-14mm}{
 \begin{tikzpicture}
 \draw[blue,in=-90,out=-90,looseness=2] (-0.5,0.5) to (-1.5,0.5);
 \draw[blue] (-1,-.1) to (-1,-.9);
 \draw[blue] (-1.5,.5) to (-1.5,1);
 \draw[blue] (-.5,.5) to (-.5,1);
 \node[draw,thick,rounded corners, fill=white,scale=.8] at (-1.5,.6) {$b$};
 \node[draw,thick,rounded corners, fill=white,scale=.8] at (-1,-.5) {$b$};
 \end{tikzpicture}} 
 \] 
 The second relation follows from Remark~\ref{ERdualrem}, since $b$ is selfdual.
 \end{proof}
 
 We will see in \S\ref{sub:ERVec} that a unital selfdual idempotent satisfying the exchange relations is not necessarily a biprojection.
 
 
 \subsection{Biprojections and exchange relations in Vec}\label{sub:ERVec}\
 
Let $A$ be a finite-dimensional $\mathbbm{k}$-algebra. A bilinear form $\kappa : A \times A \to \mathbbm{k}$ is called \emph{associative} if $\kappa(ab, c) = \kappa(a, bc)$ for all $a,b,c \in A$, and \emph{non-degenerate} if the map $A \to A^*$, $a \mapsto \kappa(a,-)$, is an isomorphism. A \emph{Frobenius algebra} structure on $A$ in $\mathbf{Vec}_{\mathbbm{k}}$ is given by a bilinear form $\kappa$ that is both associative and non-degenerate. This form endows $A$ with the structure of a counital coalgebra (e.g., see \cite{FS08}), with counit 
 \(
 \epsilon(a) = \kappa(a,1)
 \)
 and comultiplication
 \[
 \Delta(a) = \sum_i (a e_i) \otimes e_i',
 \]
 where $(e_i)$ and $(e_i')$ are dual bases of $A$ (i.e. $\kappa(e_i,e_j') = \delta_{i,j}$). 
 Let $e$ and $m$ denote the unit and multiplication morphisms from $A$ as an object of $\VVec_{\mathbbm{k}}$. The (co)evaluation morphisms are then
 \(
 \ev_A = \epsilon \circ m \) 
 and 
 \(\coev_A = \Delta \circ e.
 \)
 Hence, for every morphism $f \in \End_{\mathbbm{k}}(A)$, its dual is given by
 \[
 f^*(a) = \sum_i \epsilon\big(a f(e_i)\big) e_i'.
 \]
 Finally, the convolution product $(f * g) := m \circ (f \otimes g) \circ \Delta$ satisfies
 \[
 (f * g)(a) = \sum_i f(a e_i)\, g(e_i').
 \]
 
 \begin{lemma} \label{lem0:12345}
 Let $A$ be a Frobenius algebra in \emph{$\VVec_{\mathbbm{k}}$} with dual bases $(e_i)$ and $(e_i')$. Consider a morphism $f \in \End_{\mathbbm{k}}(A)$. Let $a_i:=f(e_i)$. Then the morphism $f$ is
 \begin{enumerate}[(1)]
 \item \label{lem0_1} \emph{Selfdual ($f^*=f$)} if and only if for all $a,b$, $\epsilon(af(b)) = \epsilon(f(a)b)$, if and only if for all $i,j$, $\epsilon(e_ia_j) = \epsilon(a_ie_j)$.
 \item \label{lem0_2} \emph{Unital ($f \circ e=e$)} if and only if $f(1)=1$.
 \item \label{lem0_3} \emph{Idempotent ($f^2=f$)} if and only if for all $i$, $f(a_i)=a_i$, 
 \item \label{lem0_4} \emph{Convolution-stable ($f * f= \lambda f$)} if and only if for all $i$, $\sum_j f(e_ie_j)f(e_j') = \lambda f(e_i)$,
 \item \label{lem0_5} \emph{ER (exchange relations)} if and only if $f(f(a)b) = f(a)f(b) = f(af(b))$ for all $a,b$, if and only if for all $i,j$, $f(a_ie_j) = a_i a_j = f(e_ia_j).$
 \end{enumerate}
 \end{lemma}
 \begin{proof}
 Straightforward. 
 \end{proof}
 
 \begin{example}
 The unital algebra $M_n(\mathbbm{k})$ is a Frobenius algebra with the bilinear form $\kappa(X,Y) = \Tr(XY)$. Hence, the counit is $\epsilon = \Tr$, the comultiplication is
 $$\Delta(X) = \sum_{i,j}(XE_{i,j})\otimes E_{j,i},$$
 where $E_{i,j}$ denote the elementary matrices in $M_n(\mathbbm{k})$, and for every $S,T \in \End_{\mathbbm{k}}(M_n(\mathbbm{k}))$, 
 $$T^*(X) = \sum_{i,j} \Tr(XT(E_{i,j}))E_{j,i},$$
 $$(S * T)(X) = \sum_{i,j} S(XE_{i,j})T(E_{j,i}).$$
 \end{example}
 \begin{lemma} \label{lem:12345}
 Consider a morphism $T \in \End_{\mathbbm{k}}(M_n(\mathbbm{k}))$. Let us denote the entries of $A_{i,j}:=T(E_{i,j})$ as $A_{i,j,k,l}$. Then the morphism $T$ is
 \begin{enumerate}[(1)]
 \item \label{lem_1} \emph{Selfdual ($T^*=T$)} if and only if for all $X,Y$, $\Tr(XT(Y)) = \Tr(T(X)Y)$, if and only if for all $i,j,k,l$, $\Tr(E_{i,j}A_{k,l}) = \Tr(A_{i,j}E_{k.l})$, if and only if $A_{i,j,k,l} = A_{l,k,j,i}.$
 \item \label{lem_2} \emph{Unital ($T(\id)=\id$)} if and only if for all $k,l$, $\sum_i A_{i,i,k,l} = \delta_{k,l},$
 \item \label{lem_3} \emph{Idempotent ($T^2=T$)} if and only if for all $i,j,k,l$, $\sum_{r,s}A_{i,j,r,s} A_{r,s,k,l} = A_{i,j,k,l}.$
 \item \label{lem_4} \emph{Convolution-stable ($T * T = \lambda T$)} if and only if for all $i,j$
 $\sum_{r} A_{i,r} A_{r,j} = \lambda A_{i,j},$
 if and only if for all $i,j,k,l$,
 $\sum_{r,s} A_{i,r,k,s} A_{r,j,s,l} = \lambda A_{i,j,k,l},$
 \item \label{lem_5} \emph{ER (exchange relations $T(T(X)Y) = T(X)T(Y) = T(XT(Y))$ for all $X,Y$)} if and only if for all $i,j,k,l$ $T(A_{i,j}E_{k,l}) = A_{i,j}A_{k,l} = T(E_{i,j}A_{k,l}),$ if and only if
 for all $i,j,k,l,r,s$, $\sum_t A_{i,j,t,k}A_{t,l,r,s} = \sum_t A_{i,j,r,t}A_{k,l,t,s} = \sum_t A_{k,l,j,t} A_{i,t,r,s}.$
 \end{enumerate}
 \end{lemma}
 \begin{proof}
 Straightforward. 
 \end{proof}
 
 \noindent Recall that a \emph{biprojection} is a selfdual unital convolution-stable idempotent (Definition \ref{def:bipro}).
 
 \begin{proposition} \label{prop:bipro}
 For all $n \le 3$, every biprojection $T \in \End_{\mathbb{C}}(M_n(\mathbb{C}))$ is ER.
 \end{proposition}
 \begin{proof}
 We are reduced to show that, in Lemma \ref{lem:12345}, the identities (\ref{lem_1}), (\ref{lem_2}), (\ref{lem_3}) and (\ref{lem_4}) implies (\ref{lem_5}). It is checked in \S \ref{sub:sage}.
 \end{proof}

The converse of Proposition~\ref{prop:bipro} fails: a selfdual unital ER idempotent need not be a biprojection (see \S\ref{sub:sage}). For $n=2$ (resp.~$n=3$), these idempotents form an algebraic set of dimension $3$ (resp. dimension $7$), versus dimension $2$ (resp. dimension $6$) for biprojections.
 
 \begin{question} \label{qu:BiER}
 Can Proposition~\ref{prop:bipro} be extended to all $n$? If so, does every biprojection in any $\mathbbm{k}$-linear Karoubian monoidal category satisfy the exchange relations?
 \end{question}
 
 \begin{example} \label{ex:bipromat}
 For any \( n \), the following are examples of biprojections \( T \in \End_{\mathbb{C}}(M_n(\mathbb{C})) \): 
 for \( \lambda = n \), the identity map \( T(X) = X \); 
 for \( \lambda = 1 \), the diagonal map \( T(X) = \sum_i X_{i,i} E_{i,i} \); 
 and for \( \lambda = 1/n \), the normalized trace map \( T(X) = \tfrac{1}{n}\Tr(X)\,\mathrm{id}. \)
 \end{example}
 
 \begin{proposition}[\S \ref{sub:sage}] \label{prop:bipro2}
 For all $n \le 3$, every biprojection $T \in \End_{\mathbb{C}}(M_n(\mathbb{C}))$ is convolution-stable with $\lambda \in \{1/n, 1, n\}$. There is a unique biprojection for $\lambda = 1/n$ and for $\lambda = n$ (see Example~\ref{ex:bipromat}). For $\lambda = 1$, the algebraic set of biprojections is continuously infinite, being $2$-dimensional for $n = 2$ (see Proposition \ref{prop:bipro3}) and $6$-dimensional for $n = 3$.
 \end{proposition} 
 
 \begin{question}
 Is it possible to extend Proposition~\ref{prop:bipro2} to all $n$?
 \end{question}
 
 \begin{proposition}\label{prop:matrix_Frob_subalg} 
 The Frobenius subalgebras of $M_2(\mathbb{C})$ are precisely $\mathbb{C}\,\id$, the conjugates of the diagonal subalgebra $\mathbb{C}E_{1,1} + \mathbb{C}E_{2,2}$, and $M_2(\mathbb{C})$ itself.
 \end{proposition}
 
 \begin{proof}
 Any proper nontrivial Frobenius subalgebra must have dimension $2$ or $3$. A $3$-dimensional unital subalgebra cannot be Frobenius, since it is conjugate to the upper triangular subalgebra, which is degenerate with respect to the bilinear form $\kappa(X,Y) = \Tr(XY)$. Indeed, the $3 \times 3$ matrix $\big(\Tr(E_{i,j}E_{k,l})\big)$, where $E_{i,j}, E_{k,l} \neq E_{2,1}$, is not invertible. Next, any non-scalar matrix $X$ generates a $2$-dimensional unital subalgebra by the Cayley--Hamilton theorem:
 \(
 X^2 - \Tr(X) X + \det(X)\,\id = 0.
 \)
 This subalgebra is non-degenerate with respect to $\kappa$ if and only if
 \(
 2\Tr(X^2) - \Tr(X)^2 \neq 0,
 \)
 which is equivalent to
 \(
 \Tr(X)^2 - 4\det(X) \neq 0,
 \)
 i.e., $X$ has two distinct eigenvalues.
 \end{proof}

 %
 
 The following proposition is included for informational purposes; the proof, which was obtained with computer assistance, is omitted.
 \begin{proposition}[Classification of the biprojections $T \in \End_{\mathbb{C}}(M_2(\mathbb{C}))$ with $\lambda=1$] \label{prop:bipro3}
 $$T(E_{i,j})=A_{i,j} \text{ where } A_{1,1}, \ A_{1,2}, \ A_{2,1}, \ A_{2,2} = $$
 \begin{enumerate}[(1)]
 \item $ \begin{pmatrix} 1 & -s \\ 0 & 0 \end{pmatrix}$, $\begin{pmatrix} 0 & 0 \\ 0 & 0 \end{pmatrix}$, $\begin{pmatrix} -s & 2s^2 \\ 0 & s \end{pmatrix}$, $\begin{pmatrix} 0 & s \\ 0 & 1 \end{pmatrix}$, with $s \in \mathbb{C}$,
 \item $\begin{pmatrix} \frac{1}{2} & 0 \\ 0 & \frac{1}{2} \end{pmatrix}$, $\begin{pmatrix} 0 & \frac{1}{2} \\ \frac{1}{4u} & 0 \end{pmatrix}$, $\begin{pmatrix} 0 & u \\ \frac{1}{2} & 0 \end{pmatrix}$, $\begin{pmatrix} \frac{1}{2} & 0 \\ 0 & \frac{1}{2} \end{pmatrix}$, with $u \in \mathbb{C} \setminus \{0\}$,
 \item $\begin{pmatrix} k & \frac{(k-1)(k-1/2)}{t} \\ -t & 1-k \end{pmatrix}$, $\begin{pmatrix} -t & 1-k \\ \frac{t^2}{k-1/2} & t \end{pmatrix}$, $\begin{pmatrix} \frac{(k-1)(k-1/2)}{t} & \frac{(k-1)^2(k-1/2)}{t^2} \\ 1-k & -\frac{(k-1)(k-1/2)}{t} \end{pmatrix}$, $\begin{pmatrix} 1-k & -\frac{(k-1)(k-1/2)}{t} \\ t & k \end{pmatrix}$, with $k \in \mathbb{C}\setminus\{1/2\}$, and $t \in \mathbb{C}\setminus\{0\}$.
 \end{enumerate}
 \end{proposition}
 Each of the first two families in Proposition~\ref{prop:bipro3} can be viewed as a limit of the last family. The images of these biprojections are the two-dimensional unital algebras generated by \(A_{1,2}\) or \(A_{2,1}\). By Proposition~\ref{prop:bipro} and Theorem~\ref{thm:main}, these are Frobenius subalgebras of \(M_2(\mathbb{C})\).

 \section{Exchange relations and Frobenius subalgebras}\label{main}
 In this section, we study unital selfdual idempotents satisfying the exchange relations in a Karoubian monoidal category, and we show that they provide a complete characterization of the Frobenius subalgebras of a Frobenius algebra (Theorem~\ref{thm:main}). We then present a connected unitary analogue (Theorem~\ref{thm:uniMain}). As an application, we characterize averaging operators on C*-correspondences in terms of intermediate C*-subalgebras arising from finite-index irreducible unital inclusions of C*-algebras (Corollary~\ref{cor:C*algthm}).
 %
 %
 \begin{thm} \label{thm:main}
 Let $X$ be a Frobenius algebra in a Karoubian monoidal category $\mcal C$. A unital selfdual idempotent $b \in \End_{\mcal C}(X)$ arises from a Frobenius subalgebra if and only if it satisfies the exchange relations.
 \end{thm}
 \begin{proof}
 The forward implication is established in \cite[Theorem~3.19]{GP25}. For completeness, we briefly recall it here. Let $Y$
 be a Frobenius subalgebra of $X$
 and let 
 \(
 i: Y \to X
 \) 
 denote the algebra monomorphism as in Definition~\ref{Frobsubalgdefn}. Setting 
 \(
 b = i \circ i^*,
 \) 
 we then obtain
 \vspace*{-2.75mm}
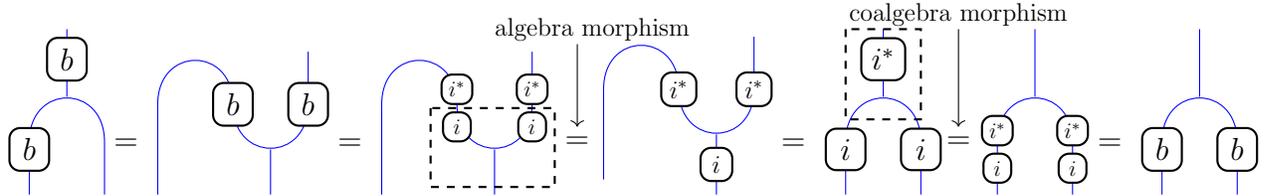
 \[ \raisebox{-6mm}{
 \begin{tikzpicture}
 \draw[blue,in=90,out=90,looseness=2] (-0.5,0.5) to (-1.5,0.5);
 \draw[blue] (-1,1.1) to (-1,2);
 \draw[blue] (-.5,.5) to (-.5,-.2);
 \draw[blue] (-1.5,.5) to (-1.5,-.2);
 \node[draw,thick,rounded corners, fill=white,minimum width=15] at (-1.5,.4){$b$};
 \node[draw,thick,rounded corners, fill=white,minimum width=15] at (-1,1.6){$b$};
 \end{tikzpicture}}
 =
 \raisebox{-6mm}{
 \begin{tikzpicture}
 \draw[blue,in=90,out=90,looseness=2] (0,0) to (1,0);
 \draw[blue,in=-90,out=-90,looseness=2] (1,0) to (2,0);
 \draw[blue] (1.5,-.6) to (1.5,-1.2);
 \draw[blue] (0,0) to (0,-1.2);
 \draw[blue] (2,0) to (2,.7);
 \node[draw,thick,rounded corners, fill=white,minimum width=15] at (1,0) {$b$};
 \node[draw,thick,rounded corners, fill=white,minimum width=15] at (2,0) {$b$};
 \end{tikzpicture}}
 =
 \raisebox{-6mm}{
 \begin{tikzpicture}
 \draw[blue,in=90,out=90,looseness=2] (0,0) to (1,0);
 \draw[blue,in=-90,out=-90,looseness=2] (1,0) to (2,0);
 \draw[blue] (1.5,-.6) to (1.5,-1.2);
 \draw[blue] (0,0) to (0,-1.2);
 \draw[blue] (2,0) to (2,.7);
 \draw[dashed,thick] (.65,-0.05) rectangle (2.3,-1.1);
 \draw[->] (2.6,.8) to (2.6,-.3);
 \node[draw,thick,rounded corners, fill=white,minimum width=15,scale=.7] at (1,.2) {$i^*$};
 \node[draw,thick,rounded corners, fill=white,minimum width=15,scale=.7] at (1,-.3) {$i$};
 \node[draw,thick,rounded corners, fill=white,minimum width=15,scale=.7] at (2,.2) {$i^*$};
 \node[draw,thick,rounded corners, fill=white,minimum width=15,scale=.7] at (2,-.3) {$i$};
 \node[scale=.8] at (2.8,1) {algebra morphism};
 \node at (2.6,-.5) {$=$};
 \end{tikzpicture}}
 \hspace*{-1.4cm}
 \raisebox{-6mm}{
 \begin{tikzpicture}
 \draw[blue,in=90,out=90,looseness=2] (0,0) to (1,0);
 \draw[blue,in=-90,out=-90,looseness=2] (1,0) to (2,0);
 \draw[blue] (1.5,-.6) to (1.5,-1.4);
 \draw[blue] (0,0) to (0,-1.2);
 \draw[blue] (2,0) to (2,.7);
 \node[draw,thick,rounded corners, fill=white,minimum width=15,scale=.8] at (1,0) {$i^*$};
 \node[draw,thick,rounded corners, fill=white,minimum width=15,scale=.8] at (2,0) {$i^*$};
 \node[draw,thick,rounded corners, fill=white,minimum width=15,scale=.8] at (1.5,-1) {$i$};
 \end{tikzpicture}}
 =
 \raisebox{-6mm}{
 \begin{tikzpicture}
 \draw[blue,in=90,out=90,looseness=2] (-0.5,0.5) to (-1.5,0.5);
 \draw[blue] (-1,1.1) to (-1,2);
 \draw[blue] (-.5,.5) to (-.5,-.2);
 \draw[blue] (-1.5,.5) to (-1.5,-.2);
 \draw[dashed,thick] (-1.5,2) rectangle (-.5,.8);
 \draw[->] (0,2) to (0,.7);
 \node[draw,thick,rounded corners, fill=white,minimum width=15] at (-.5,.4){$i$};
 \node[draw,thick,rounded corners, fill=white,minimum width=15] at (-1.5,.4){$i$};
 \node[draw,thick,rounded corners, fill=white,minimum width=15] at (-1,1.6){$i^*$};
 \node[scale=.8] at (0,2.2) {coalgebra morphism};
 \node at (0,.5) {$=$};
 \end{tikzpicture}}
 \hspace*{-1.4cm}
 \raisebox{-6mm}{
 \begin{tikzpicture}
 \draw[blue,in=90,out=90,looseness=2] (-0.5,0.5) to (-1.5,0.5);
 \draw[blue] (-1,1.1) to (-1,2);
 \draw[blue] (-.5,.5) to (-.5,-.2);
 \draw[blue] (-1.5,.5) to (-1.5,-.2);
 \node[draw,thick,rounded corners, fill=white,minimum width=15,scale=.7] at (-1.5,.65) {$i^*$};
 \node[draw,thick,rounded corners, fill=white,minimum width=15,scale=.7] at (-1.5,.15) {$i$};
 \node[draw,thick,rounded corners, fill=white,minimum width=15,scale=.7] at (-.5,.65) {$i^*$};
 \node[draw,thick,rounded corners, fill=white,minimum width=15,scale=.7] at (-.5,.15) {$i$};
 \end{tikzpicture}}
 =
 \raisebox{-6mm}{
 \begin{tikzpicture}
 \draw[blue,in=90,out=90,looseness=2] (-0.5,0.5) to (-1.5,0.5);
 \draw[blue] (-1,1.1) to (-1,2);
 \draw[blue] (-.5,.5) to (-.5,-.2);
 \draw[blue] (-1.5,.5) to (-1.5,-.2);
 \node[draw,thick,rounded corners, fill=white,minimum width=15] at (-1.5,.4){$b$};
 \node[draw,thick,rounded corners, fill=white,minimum width=15] at (-.5,.4){$b$};
 \end{tikzpicture}} 
 \]
 The other equation can be derived in a similar manner.
 
 Conversely, suppose 
 \(
 b : X \to X
 \) 
 is a unital selfdual idempotent satisfying the exchange relations. Because $\mathcal{C}$ is Karoubian (i.e., idempotent-complete), it yields a monomorphism 
 \(
 u : Y \to X
 \) 
 and an epimorphism 
 \(
 v : X \to Y
 \) 
 such that $b = u \circ v$ and $v \circ u = \id_Y$. We will equip $Y$ with a Frobenius algebra structure and show that the monomorphism $u$ realizes it as a Frobenius subalgebra in the sense of Definition~\ref{Frobsubalgdefn}. 
 
 Define the multiplication 
 \(
 m_Y : Y \otimes Y \to Y
 \) 
 and the unit 
 \(
 e_Y : \mathbbm{1} \to Y
 \) 
 as follows:
 \vspace*{-1mm}
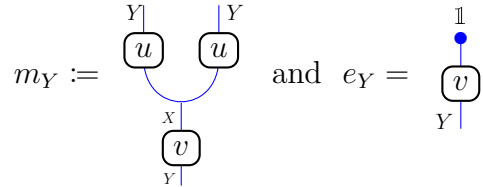
 \[ m_Y \coloneqq \raisebox{-14mm}{
 \begin{tikzpicture}
 \draw[blue,in=-90,out=-90,looseness=2] (-0.5,0.5) to (-1.5,0.5);
 \draw[blue] (-1,-.1) to (-1,-.6);
 \draw[blue] (-1.5,.6) to (-1.5,1.2);
 \draw[blue] (-.5,.6) to (-.5,1.2);
 \draw[blue] (-1, -.8) to (-1,-1.2);
 \node[draw,thick,rounded corners, fill=white] at (-1.5,.6) {$u$};
 \node[draw,thick,rounded corners, fill=white] at (-.5,.6) {$u$};
 \node[draw,thick,rounded corners, fill=white] at (-1,-.7) {$v$};
 \node[left,scale=0.5] at (-1,-.3) {$X$};
 \node[left,scale=0.7] at (-1.4,1.1) {$Y$};
 \node[right,scale=0.7] at (-.5,1.1) {$Y$};
 \node[left,scale=0.5] at (-1,-1.1) {$Y$};
 \end{tikzpicture}} \ \ \text{and} \ \ 
 e_Y = \raisebox{-7mm}{
 \begin{tikzpicture}
 \draw [blue] (-0.8,-.6) to (-.8,.6);
 \node[draw,thick,rounded corners, fill=white] at (-.8,0) {$v$};
 \node at (-.8,.6) {${\color{blue}\bullet}$};
 \node[scale=.8] at (-.8,.9) {$\mathbbm{1}$};
 \node[left,scale=0.7] at (-.8,-.5) {$Y$};
 \end{tikzpicture}}\] 
 
 Define $\delta_Y : Y \to Y \ot Y$ and $\epsilon_Y : Y \to \mathbbm{1}$ as follows :
 \vspace*{-2mm}
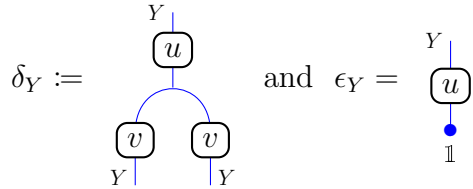
 \[ \delta_Y \coloneqq \raisebox{-14mm}{
 \begin{tikzpicture}
 \draw[blue,in=90,out=90,looseness=2] (-0.5,0.5) to (-1.5,0.5);
 \draw[blue] (-1,1.1) to (-1,2.1);
 \draw[blue] (-.5,.4) to (-.5,-.2);
 \draw[blue] (-1.5,.4) to (-1.5,-.2);
 \node[draw,thick,rounded corners, fill=white] at (-1,1.6) {$u$};
 \node[draw,thick,rounded corners, fill=white] at (-.5,.4) {$v$};
 \node[draw,thick,rounded corners, fill=white] at (-1.5,.4) {$v$};
 \node[left,scale=0.7] at (-1,2.1) {$Y$};
 \node[left,scale=0.7] at (-1.5,-.1) {$Y$};
 \node[right,scale=0.7] at (-.5,-.1) {$Y$};
 \end{tikzpicture}} \ \ \text{and} \ \ 
 \epsilon_Y = \raisebox{-11mm}{
 \begin{tikzpicture}
 \draw [blue] (-0.8,-.6) to (-.8,.6);
 \node at (-.8,-.6) {${\color{blue}\bullet}$};
 \node[draw,thick,rounded corners, fill=white] at (-.8,0) {$u$};
 \node[left,scale=0.7] at (-.8,.6) {$Y$};
 \node[scale=.8] at (-.8,-.9) {$\mathbbm{1}$};
 \end{tikzpicture}} \]

 \begin{lem}\label{FrobAlgY}
 $(Y,m_Y,e_Y,\delta_Y,\epsilon_Y)$ is a Frobenius algebra in $\mathcal{C}$.
 \end{lem}
 \begin{proof} Let us first establish \emph{unitality}:
 %
 \vspace*{-8mm}
 \[m_Y \circ (e_Y \ot 1_Y) = \raisebox{-14mm}{
 \begin{tikzpicture}
 \draw[blue,in=-90,out=-90,looseness=2] (-0.5,0.5) to (-1.5,0.5);
 \draw[blue] (-1,-.1) to (-1,-.6);
 \draw[blue] (-1.5,.6) to (-1.5,1.2);
 \draw[blue] (-.5,.6) to (-.5,1.2);
 \draw[blue] (-1, -.8) to (-1,-1.4);
 \node[draw,thick,rounded corners, fill=white] at (-1.5,.6) {$b$};
 \node[draw,thick,rounded corners, fill=white] at (-.5,.6) {$u$};
 \node[draw,thick,rounded corners, fill=white] at (-1,-.7) {$v$};
 \node at (-1.5,1.2) {${\color{blue}\bullet}$};
 \end{tikzpicture}} = 
 \raisebox{-18mm}{
 \begin{tikzpicture}
 \draw[blue,in=-90,out=-90,looseness=2] (-0.5,0.5) to (-1.5,0.5);
 \draw[blue] (-1,-.1) to (-1,-.6);
 \draw[blue] (-1.5,.6) to (-1.5,2);
 \draw[blue] (-.5,.6) to (-.5,2);
 \draw[blue] (-1, -.8) to (-1,-1.4);
 \node[draw,thick,rounded corners, fill=white] at (-1.5,.6) {$b$};
 \node[draw,thick,rounded corners, fill=white] at (-.5,.6) {$b$};
 \node[draw,thick,rounded corners, fill=white] at (-.5,1.4) {$u$};
 \node[draw,thick,rounded corners, fill=white] at (-1,-.7) {$v$};
 \node at (-1.5,2) {${\color{blue}\bullet}$};
 \end{tikzpicture}}
 =
 \raisebox{-18mm}{
 \begin{tikzpicture}
 \draw[blue,in=-90,out=-90,looseness=2] (-0.5,0.5) to (-1.5,0.5);
 \draw[blue] (-1,-.1) to (-1,-.6);
 \draw[blue] (-1.5,.5) to (-1.5,2);
 \draw[blue] (-.5,.6) to (-.5,2);
 \draw[blue] (-1, -.8) to (-1,-1.8);
 \node[draw,thick,rounded corners, fill=white] at (-.5,.6) {$b$};
 \node[draw,thick,rounded corners, fill=white] at (-.5,1.4) {$u$};
 \node[draw,thick,rounded corners, fill=white] at (-1,-.7) {$b$};
 \node[draw,thick,rounded corners, fill=white] at (-1,-1.4) {$v$};
 \node at (-1.5,2) {${\color{blue}\bullet}$};
 \end{tikzpicture}}
 = 
 \raisebox{-18mm}{
 \begin{tikzpicture}
 \draw[blue,in=-90,out=-90,looseness=2] (-0.5,0.5) to (-1.5,0.5);
 \draw[blue] (-1,-.1) to (-1,-1.4);
 \draw[blue] (-1.5,.5) to (-1.5,2);
 \draw[blue] (-.5,.5) to (-.5,2);
 \node[draw,thick,rounded corners, fill=white] at (-.5,1.4) {$u$};
 \node[draw,thick,rounded corners, fill=white] at (-1,-.8) {$v$};
 \node at (-1.5,2) {${\color{blue}\bullet}$};
 \end{tikzpicture}}
 = v \circ u = \id_Y.
 \vspace*{-1mm}\]
 The second and fourth equalities follow from the identities $b \circ u = u$ and $v \circ b = v$. 
 The third equality follows from the exchange relation for $b$, and the fifth from the unitality of $X$; the other unitality identity is analogous. We now establish \emph{associativity}:
 \vspace*{-2mm}
 \[ 
 m_Y \circ (m_Y \ot 1_Y) = \raisebox{-8mm}{\begin{tikzpicture}
 \draw[blue,in=-90,out=-90,looseness=2] (0,0) to (1,0);
 \draw[blue,in=-90,out=-90,looseness=2] (.5,-.6) to (1.5,-.6);
 \draw[blue] (1.5,-.6) to (1.5,.5);
 \draw[blue] (0,0) to (0,.5);
 \draw[blue] (1,0) to (1,.5);
 \draw[blue] (1,-1.2) to (1,-2);
 \node[draw,thick,rounded corners, fill=white] at (0,0){$u$};
 \node[draw,thick,rounded corners, fill=white] at (1,0){$u$};
 \node[draw,thick,rounded corners, fill=white,scale=.7] at (.55,-.9) {$b$};
 \node[draw,thick,rounded corners, fill=white] at (1.5,-.6) {$u$};
 \node[draw,thick,rounded corners, fill=white] at (1,-1.6) {$v$};
 \end{tikzpicture}}
 = \raisebox{-8mm}{\begin{tikzpicture}
 \draw[blue,in=-90,out=-90,looseness=2] (0,0) to (1,0);
 \draw[blue,in=-90,out=-90,looseness=2] (.5,-.6) to (1.5,-.6);
 \draw[blue] (1.5,-.6) to (1.5,.5);
 \draw[blue] (0,0) to (0,.5);
 \draw[blue] (1,0) to (1,.5);
 \draw[blue] (1,-1.2) to (1,-2);
 \node[draw,thick,rounded corners, fill=white] at (0,0){$u$};
 \node[draw,thick,rounded corners, fill=white] at (1,0){$u$};
 \node[draw,thick,rounded corners, fill=white,scale=.7] at (.55,-.9) {$b$};
 \node[draw,thick,rounded corners, fill=white,scale=.7] at (1.5,-.9) {$b$};
 \node[draw,thick,rounded corners, fill=white] at (1.6,0) {$u$};
 \node[draw,thick,rounded corners, fill=white] at (1,-1.6) {$v$};
 \end{tikzpicture}}
 = 
 \raisebox{-11mm}{\begin{tikzpicture}
 \draw[blue,in=-90,out=-90,looseness=2] (0,0) to (1,0);
 \draw[blue,in=-90,out=-90,looseness=2] (.5,-.6) to (1.5,-.6);
 \draw[blue] (1.5,-.6) to (1.5,.5);
 \draw[blue] (0,0) to (0,.5);
 \draw[blue] (1,0) to (1,.5);
 \draw[blue] (1,-1.2) to (1,-2.7);
 \node[draw,thick,rounded corners, fill=white] at (0,0){$u$};
 \node[draw,thick,rounded corners, fill=white] at (1,0){$u$};
 \node[draw,thick,rounded corners, fill=white,scale=.7] at (1,-1.6) {$b$};
 \node[draw,thick,rounded corners, fill=white,scale=.7] at (1.5,-.9) {$b$};
 \node[draw,thick,rounded corners, fill=white] at (1.6,0) {$u$};
 \node[draw,thick,rounded corners, fill=white] at (1,-2.3) {$v$};
 \end{tikzpicture}}
 =
 \raisebox{-8mm}{\begin{tikzpicture}
 \draw[blue,in=-90,out=-90,looseness=2] (0,0) to (1,0);
 \draw[blue,in=-90,out=-90,looseness=2] (.5,-.6) to (1.5,-.6);
 \draw[blue] (1.5,-.6) to (1.5,.5);
 \draw[blue] (0,0) to (0,.5);
 \draw[blue] (1,0) to (1,.5);
 \draw[blue] (1,-1.2) to (1,-2);
 \node[draw,thick,rounded corners, fill=white] at (0,0){$u$};
 \node[draw,thick,rounded corners, fill=white] at (1,0){$u$};
 \node[draw,thick,rounded corners, fill=white] at (1.6,0) {$u$};
 \node[draw,thick,rounded corners, fill=white] at (1,-1.6) {$v$};
 \end{tikzpicture}}
 \vspace*{-2mm}
 \]
 These equalities follow as for the unitality. The first associativity identity for $Y$ follows from that for $X$, and the second is similar.
 The proofs of the \emph{co-unitality} and \emph{co-associativity} of $Y$ are entirely analogous. 
 It remains to verify the Frobenius condition:
 %
 \vspace*{-2mm}
 \begin{equation}\label{frobenius1}
 \delta_Y \circ m_Y = \raisebox{-11mm}{
 \begin{tikzpicture}
 \draw[blue,in=90,out=90,looseness=2] (0,0) to (1,0);
 \draw[blue] (.5,.6) to (.5,1.2);
 \draw[blue,in=-90,out=-90,looseness=2] (0,1.8) to (1,1.8);
 \draw[blue] (0,0) to (0,-.4);
 \draw[blue] (1,0) to (1,-.4);
 \draw[blue] (0,1.8) to (0,2.2);
 \draw[blue] (1,1.8) to (1,2.2);
 \node[draw,thick,rounded corners, fill=white] at (0,0) {$v$};
 \node[draw,thick,rounded corners, fill=white] at (1,0) {$v$};
 \node[draw,thick,rounded corners, fill=white] at (0,1.8) {$u$};
 \node[draw,thick,rounded corners, fill=white] at (1,1.8) {$u$};
 \node[draw,thick,rounded corners, fill=white,scale=.7] at (.5,.9) {$b$};
 \end{tikzpicture}}
 =
 \raisebox{-11mm}{
 \begin{tikzpicture}
 \draw[blue,in=90,out=90,looseness=2] (0,0) to (1,0);
 \draw[blue] (.5,.6) to (.5,1.2);
 \draw[blue,in=-90,out=-90,looseness=2] (0,1.8) to (1,1.8);
 \draw[blue] (0,0) to (0,-.4);
 \draw[blue] (1,0) to (1,-.8);
 \draw[blue] (0,1.8) to (0,2.2);
 \draw[blue] (1,1.8) to (1,2.2);
 \node[draw,thick,rounded corners, fill=white] at (0,0) {$v$};
 \node[draw,thick,rounded corners, fill=white,scale=.7] at (1,-.5) {$v$};
 \node[draw,thick,rounded corners, fill=white] at (0,1.8) {$u$};
 \node[draw,thick,rounded corners, fill=white] at (1,1.8) {$u$};
 \node[draw,thick,rounded corners, fill=white,scale=.7] at (.5,.9) {$b$};
 \node[draw,thick,rounded corners, fill=white,scale=.7] at (1,.05) {$b$};
 \end{tikzpicture}}
 =
 \raisebox{-11mm}{
 \begin{tikzpicture}
 \draw[blue,in=90,out=90,looseness=2] (0,0) to (1,0);
 \draw[blue] (.5,.6) to (.5,1.2);
 \draw[blue,in=-90,out=-90,looseness=2] (0,1.8) to (1,1.8);
 \draw[blue] (0,0) to (0,-.8);
 \draw[blue] (1,0) to (1,-.8);
 \draw[blue] (0,1.8) to (0,2.2);
 \draw[blue] (1,1.8) to (1,2.2);
 \node[draw,thick,rounded corners, fill=white,scale=.7] at (0,-.5) {$v$};
 \node[draw,thick,rounded corners, fill=white,scale=.7] at (1,-.5) {$v$};
 \node[draw,thick,rounded corners, fill=white] at (0,1.8) {$u$};
 \node[draw,thick,rounded corners, fill=white] at (1,1.8) {$u$};
 \node[draw,thick,rounded corners, fill=white,scale=.7] at (1,.05) {$b$};
 \node[draw,thick,rounded corners, fill=white,scale=.7] at (0,.05) {$b$};
 \end{tikzpicture}}
 =
 \raisebox{-11mm}{
 \begin{tikzpicture}
 \draw[blue,in=90,out=90,looseness=2] (0,0) to (1,0);
 \draw[blue] (.5,.6) to (.5,1.2);
 \draw[blue,in=-90,out=-90,looseness=2] (0,1.8) to (1,1.8);
 \draw[blue] (0,0) to (0,-.4);
 \draw[blue] (1,0) to (1,-.4);
 \draw[blue] (0,1.8) to (0,2.2);
 \draw[blue] (1,1.8) to (1,2.2);
 \node[draw,thick,rounded corners, fill=white] at (0,0) {$v$};
 \node[draw,thick,rounded corners, fill=white] at (1,0) {$v$};
 \node[draw,thick,rounded corners, fill=white] at (0,1.8) {$u$};
 \node[draw,thick,rounded corners, fill=white] at (1,1.8) {$u$};
 \end{tikzpicture}}
 \end{equation}
 %
 \noindent These equalities follow as for the unitality. Pictorially, $(m_Y \otimes 1_Y) \circ (1_Y \otimes \delta_Y)$ is:
 \vspace*{-2mm}
 \begin{equation}\label{frobenius2}
 \raisebox{-11mm}{
 \begin{tikzpicture}
 \draw[blue,in=-90,out=-90,looseness=2] (0,0) to (1,0);
 \draw[blue,in=90,out=90,looseness=2] (1,0) to (2,0);
 \draw[blue] (.5,-.6) to (.5,-1.5);
 \draw[blue] (1.5,.6) to (1.5,1.5);
 \draw[blue] (0,0) to (0,1.5);
 \draw[blue] (2,0) to (2,-1.5);
 \node[draw,thick,rounded corners, fill=white] at (0,1) {$u$};
 \node[draw,thick,rounded corners, fill=white] at (1.5,1) {$u$};
 \node[draw,thick,rounded corners, fill=white] at (.5,-1) {$v$};
 \node[draw,thick,rounded corners, fill=white] at (2,-1) {$v$};
 \node[draw,thick,rounded corners, fill=white,scale=.7] at (1,0) {$b$};
 \end{tikzpicture}}
 =
 \raisebox{-11mm}{
 \begin{tikzpicture}
 \draw[blue,in=-90,out=-90,looseness=2] (0,0) to (1,0);
 \draw[blue,in=90,out=90,looseness=2] (1,0) to (2,0);
 \draw[blue] (.5,-.6) to (.5,-1.5);
 \draw[blue] (1.5,.6) to (1.5,1.5);
 \draw[blue] (0,0) to (0,1.5);
 \draw[blue] (2,0) to (2,-1.5);
 \node[draw,thick,rounded corners, fill=white] at (0,1) {$u$};
 \node[draw,thick,rounded corners, fill=white] at (1.5,1) {$u$};
 \node[draw,thick,rounded corners, fill=white] at (.5,-1) {$v$};
 \node[draw,thick,rounded corners, fill=white] at (2,-1) {$v$};
 \node[draw,thick,rounded corners, fill=white,scale=.7] at (1,0) {$b$};
 \node[draw,thick,rounded corners, fill=white,scale=.7] at (2,0) {$b$};
 \end{tikzpicture}}
 =
 \raisebox{-11mm}{
 \begin{tikzpicture}
 \draw[blue,in=-90,out=-90,looseness=2] (0,0) to (1,0);
 \draw[blue,in=90,out=90,looseness=2] (1,0) to (2,0);
 \draw[blue] (.5,-.6) to (.5,-1.5);
 \draw[blue] (1.5,.6) to (1.5,1.7);
 \draw[blue] (0,0) to (0,1.5);
 \draw[blue] (2,0) to (2,-1.5);
 \node[draw,thick,rounded corners, fill=white] at (0,1) {$u$};
 \node[draw,thick,rounded corners, fill=white,scale=.7] at (1.5,1.4) {$u$};
 \node[draw,thick,rounded corners, fill=white] at (.5,-1) {$v$};
 \node[draw,thick,rounded corners, fill=white] at (2,-1) {$v$};
 \node[draw,thick,rounded corners, fill=white,scale=.7] at (1.5,.9) {$b$};
 \node[draw,thick,rounded corners, fill=white,scale=.7] at (2,0) {$b$};
 \end{tikzpicture}}
 =
 \raisebox{-11mm}{
 \begin{tikzpicture}
 \draw[blue,in=-90,out=-90,looseness=2] (0,0) to (1,0);
 \draw[blue,in=90,out=90,looseness=2] (1,0) to (2,0);
 \draw[blue] (.5,-.6) to (.5,-1.5);
 \draw[blue] (1.5,.6) to (1.5,1.5);
 \draw[blue] (0,0) to (0,1.5);
 \draw[blue] (2,0) to (2,-1.5);
 \node[draw,thick,rounded corners, fill=white] at (0,1) {$u$};
 \node[draw,thick,rounded corners, fill=white] at (1.5,1) {$u$};
 \node[draw,thick,rounded corners, fill=white] at (.5,-1) {$v$};
 \node[draw,thick,rounded corners, fill=white] at (2,-1) {$v$};
 \end{tikzpicture}} 
 \end{equation}
 These equalities follow as for the unitality. Similarly, we obtain:
 \vspace*{-2mm}
 \begin{equation}\label{frobenius3}
 (1_Y \ot m_Y) \circ (\delta_Y \ot 1_Y)
 =
 \raisebox{-12mm}{
 \begin{tikzpicture}
 \draw[blue,in=90,out=90,looseness=2] (0,0) to (1,0);
 \draw[blue,in=-90,out=-90,looseness=2] (1,0) to (2,0);
 \draw[blue] (.5,.6) to (.5,1.5);
 \draw[blue] (1.5,-.6) to (1.5,-1.5);
 \draw[blue] (0,0) to (0,-1.5);
 \draw[blue] (2,0) to (2,1.5);
 \node[draw,thick,rounded corners, fill=white] at (.5,1.1) {$u$};
 \node[draw,thick,rounded corners, fill=white] at (2,1.1) {$u$};
 \node[draw,thick,rounded corners, fill=white] at (0,-1) {$v$};
 \node[draw,thick,rounded corners, fill=white] at (1.5,-1) {$v$};
 \end{tikzpicture}}
 \end{equation}
 From \Cref{frobenius1,frobenius2,frobenius3}, the Frobenius condition for $Y$ follows from that for $X$.
 \end{proof} 
 
 \begin{lem}\label{subalglem}
 The monomorphism $u$ realizes $Y$ as a Frobenius subalgebra of $X$.
 \end{lem}
 
 \begin{proof}
 By Lemma~\ref{FrobAlgY}, $Y$ is a Frobenius algebra. Since $v \circ u = \mathrm{id}_Y$, it remains to show $u^* = v$, $v^* = u$, and that $u$ is a unital algebra morphism. It is an \emph{algebra morphism} since:
 \vspace*{-2mm}
 \[ u \circ m_Y = \raisebox{-11mm}{
 \begin{tikzpicture}
 \draw[blue,in=-90,out=-90,looseness=2] (-0.5,0.5) to (-1.5,0.5);
 \draw[blue] (-1,-.1) to (-1,-.6);
 \draw[blue] (-1.5,.6) to (-1.5,1.2);
 \draw[blue] (-.5,.6) to (-.5,1.2);
 \draw[blue] (-1, -.8) to (-1,-1.2);
 \node[draw,thick,rounded corners, fill=white] at (-1.5,.6) {$u$};
 \node[draw,thick,rounded corners, fill=white] at (-.5,.6) {$u$};
 \node[draw,thick,rounded corners, fill=white,scale=.7] at (-1,-.75) {$b$};
 \node[left,scale=0.6] at (-1,-.35) {$X$};
 \node[left,scale=0.7] at (-1.4,1.2) {$Y$};
 \node[right,scale=0.7] at (-.5,1.2) {$Y$};
 \node[left, scale=.6] at (-1,-1.2) {$X$};
 \end{tikzpicture}}
 =
 \raisebox{-8mm}{
 \begin{tikzpicture}
 \draw[blue,in=-90,out=-90,looseness=2] (-0.5,0.5) to (-1.5,0.5);
 \draw[blue] (-1,-.1) to (-1,-1);
 \draw[blue] (-1.5,.5) to (-1.5,1.2);
 \draw[blue] (-.5,.5) to (-.5,1.2);
 \node[draw,thick,rounded corners, fill=white] at (-1.5,.8) {$u$};
 \node[draw,thick,rounded corners, fill=white] at (-.5,.8) {$u$};
 \node[draw,thick,rounded corners, fill=white,scale=.7] at (-1,-.5) {$b$};
 \node[draw,thick,rounded corners, fill=white,scale=.7] at (-.55,.2) {$b$};
 \node[left,scale=0.7] at (-1.4,1.3) {$Y$};
 \node[right,scale=0.7] at (-.5,1.3) {$Y$};
 \end{tikzpicture}}
 =
 \raisebox{-8mm}{
 \begin{tikzpicture}
 \draw[blue,in=-90,out=-90,looseness=2] (-0.5,0.5) to (-1.5,0.5);
 \draw[blue] (-1,-.1) to (-1,-1);
 \draw[blue] (-1.5,.5) to (-1.5,1.2);
 \draw[blue] (-.5,.5) to (-.5,1.2);
 \node[draw,thick,rounded corners, fill=white] at (-1.5,.8) {$u$};
 \node[draw,thick,rounded corners, fill=white] at (-.5,.8) {$u$};
 \node[draw,thick,rounded corners, fill=white,scale=.7] at (-1.5,.2) {$b$};
 \node[draw,thick,rounded corners, fill=white,scale=.7] at (-.55,.2) {$b$};
 \node[left,scale=0.7] at (-1.4,1.3) {$Y$};
 \node[right,scale=0.7] at (-.5,1.3) {$Y$};
 \end{tikzpicture}}
 =
 \raisebox{-8mm}{
 \begin{tikzpicture}
 \draw[blue,in=-90,out=-90,looseness=2] (-0.5,0.5) to (-1.5,0.5);
 \draw[blue] (-1,-.1) to (-1,-1);
 \draw[blue] (-1.5,.5) to (-1.5,1.2);
 \draw[blue] (-.5,.5) to (-.5,1.2);
 \node[draw,thick,rounded corners, fill=white] at (-1.5,.8) {$u$};
 \node[draw,thick,rounded corners, fill=white] at (-.5,.8) {$u$};
 \node[left,scale=0.7] at (-1.4,1.3) {$Y$};
 \node[right,scale=0.7] at (-.5,1.3) {$Y$};
 \end{tikzpicture}}
 =
 m_X \circ (u \ot u) 
 \]
 These equalities follow as for the unitality. Next, we show that $u$ is a \emph{unital morphism}: by the definition of $e_Y$ and the unitality of $b$, we obtain
 \vspace*{-2mm}
 \[ u \circ e_Y = \raisebox{-6mm}{
 \begin{tikzpicture}
 \draw [blue] (-0.8,-.6) to (-.8,.6);
 \node[draw,thick,rounded corners, fill=white] at (-.8,0) {$b$};
 \node at (-.8,.6) {${\color{blue}\bullet}$};
 \node[left,scale=.8] at (-.8,.6) {$\mathbbm{1}$};
 \node[left,scale=0.7] at (-.8,-.5) {$X$};
 \end{tikzpicture}} 
 =
 \raisebox{-6mm}{
 \begin{tikzpicture}
 \draw [blue] (-0.8,-.6) to (-.8,.6);
 \node at (-.8,.6) {${\color{blue}\bullet}$};
 \node[left,scale=.8] at (-.8,.6) {$\mathbbm{1}$};
 \node[left,scale=0.7] at (-.8,-.5) {$X$};
 \end{tikzpicture}}\] 
 Finally, we verify that $u^* = v$ and $v^* = u$:
 \vspace*{-3mm}
 \[u^* = \raisebox{-11mm}{
 \begin{tikzpicture}
 \draw[blue,in=-90,out=-90,looseness=2] (0,0) to (1,0);
 \draw[blue,in=90,out=90,looseness=2] (1,0) to (2,0);
 \draw[blue] (.5,-.6) to (.5,-1);
 \draw[blue] (1.5,.6) to (1.5,1);
 \draw[blue] (0,0) to (0,1);
 \draw[blue] (2,0) to (2,-1);
 \node[draw,thick,rounded corners, fill=white,scale=1] at (1,0) {$u$};
 \node at (.5,-1) {${\color{blue}\bullet}$};
 \node at (1.5,1) {${\color{blue}\bullet}$};
 \node[left,scale=.5] at (1.25,-.4) {$X$};
 \node[left,scale=.5] at (1.15,.4) {$Y$};
 \end{tikzpicture}}
 =
 \raisebox{-11mm}{
 \begin{tikzpicture}
 \draw[blue,in=-90,out=-90,looseness=2] (0,0) to (1,0);
 \draw[blue,in=90,out=90,looseness=2] (1,0) to (2,0);
 \draw[blue] (.5,-.6) to (.5,-1);
 \draw[blue] (1.5,.6) to (1.5,1.5);
 \draw[blue] (0,0) to (0,1);
 \draw[blue] (2,0) to (2,-1);
 \node[draw,thick,rounded corners, fill=white,scale=.8] at (1,-.25) {$u$};
 \node[draw,thick,rounded corners, fill=white,scale=.8] at (2,0.3) {$v$};
 \node[draw,thick,rounded corners, fill=white,scale=.8] at (1.05,0.3) {$v$};
 \node[draw,thick,rounded corners, fill=white,scale=.8] at (1.5,1) {$b$};
 \node at (.5,-1) {${\color{blue}\bullet}$};
 \node at (1.5,1.5) {${\color{blue}\bullet}$};
 \end{tikzpicture}}
 =
 \raisebox{-11mm}{
 \begin{tikzpicture}
 \draw[blue,in=-90,out=-90,looseness=2] (0,0) to (1,0);
 \draw[blue,in=90,out=90,looseness=2] (1,0) to (2,0);
 \draw[blue] (.5,-.6) to (.5,-1);
 \draw[blue] (1.5,.6) to (1.5,1.5);
 \draw[blue] (0,0) to (0,1);
 \draw[blue] (2,0) to (2,-1);
 \node[draw,thick,rounded corners, fill=white,scale=.8] at (2,0) {$v$};
 \node[draw,thick,rounded corners, fill=white,scale=.8] at (1.05,0) {$b$};
 \node[draw,thick,rounded corners, fill=white,scale=.8] at (1.5,1) {$b$};
 \node at (.5,-1) {${\color{blue}\bullet}$};
 \node at (1.5,1.5) {${\color{blue}\bullet}$};
 \end{tikzpicture}}
 =
 \raisebox{-11mm}{
 \begin{tikzpicture}
 \draw[blue,in=-90,out=-90,looseness=2] (0,0) to (1,0);
 \draw[blue,in=90,out=90,looseness=2] (1,0) to (2,0);
 \draw[blue] (.5,-.6) to (.5,-1);
 \draw[blue] (1.5,.6) to (1.5,1.5);
 \draw[blue] (0,0) to (0,1);
 \draw[blue] (2,0) to (2,-1);
 \node[draw,thick,rounded corners, fill=white,scale=.8] at (1.5,1) {$b$};
 \node[draw,thick,rounded corners, fill=white,scale=.8] at (2,0) {$b$};
 \node[draw,thick,rounded corners, fill=white,scale=.8] at (2,-.6) {$v$};
 \node at (.5,-1) {${\color{blue}\bullet}$};
 \node at (1.5,1.5) {${\color{blue}\bullet}$};
 \end{tikzpicture}}
 = v
 \vspace*{-2mm}
 \]
 The first equality follows from $\coev_Y = \delta_Y \circ e_Y$ and $\ev_X = \epsilon_X \circ m_X$. 
 The second equality follows from the definitions of $\delta_Y$ and $e_Y$, while the fourth is the exchange relation. 
 Finally, the unitality of $b$, together with the unitality and Frobenius property of $X$, gives the last equality. 
 A similar argument shows that $v^* = u$, completing the proof of the lemma.
 \end{proof}
 \noindent The forward implication now follows from Lemma~\ref{subalglem}.
 \end{proof}
 
 
 By Lemma~\ref{ERidemlem} below, the idempotency assumption in Theorem~\ref{thm:main} can be omitted.
 
 \begin{lem}\label{ERidemlem}
 Let $X$ be a Frobenius algebra in a monoidal category $\mathcal{C}$, and let $b : X \to X$ be a morphism satisfying the exchange relations. 
 If $b$ is unital, then $b$ is idempotent. 
 Moreover, if $\mathcal{C}$ is linear, $X$ is connected, and $b$ is nonzero, 
 then the converse holds.
 
 \end{lem}
 
 \begin{proof}
 By exchange relations and the unitality of $X$:
 \vspace*{-3mm}
 \[\raisebox{-14mm}{
 \begin{tikzpicture}
 \draw[blue,in=-90,out=-90,looseness=2] (-0.5,0.5) to (-1.5,0.5);
 \draw[blue] (-1,-.1) to (-1,-1);
 \draw[blue] (-1.5,.6) to (-1.5,1.2);
 \draw[blue] (-.5,.6) to (-.5,1.2);
 \node[draw,thick,rounded corners, fill=white] at (-1.5,.6) {$b$};
 \node[draw,thick,rounded corners, fill=white] at (-.5,.6) {$b$};
 \node at (-1.5,1.2) {${\color{blue}\bullet}$};
 \end{tikzpicture}}
 =
 \raisebox{-14mm}{
 \begin{tikzpicture}
 \draw[blue,in=-90,out=-90,looseness=2] (-0.5,0.5) to (-1.5,0.5);
 \draw[blue] (-1,-.1) to (-1,-1);
 \draw[blue] (-1.5,.5) to (-1.5,1.2);
 \draw[blue] (-.5,.5) to (-.5,1.2);
 \node[draw,thick,rounded corners, fill=white] at (-1,-.5) {$b$};
 \node[draw,thick,rounded corners, fill=white] at (-.5,.6) {$b$};
 \node at (-1.5,1.2) {${\color{blue}\bullet}$};
 \end{tikzpicture}}
 = b^2
 \vspace*{-1mm} \]
 If $b$ is unital, it follows that $b = b^2$, hence $b$ is idempotent. Conversely, if $\mathcal{C}$ is linear and $X$ is connected (that is, $\mathrm{Hom}_{\mathcal{C}}(\mathbbm{1}, X) \cong \mathbbm{k}$), then $b \circ e_X = c_b \, e_X$ for some scalar $c_b$. The pictorial identity above then implies $c_b b = b^2$. If $b^2 = b$, it follows that $c_b = 1$, and thus $b$ is unital.
 \end{proof}
 
 \begin{cor}\label{connectedsubalgcor}
 Let $X$ be a connected Frobenius algebra in a linear Karoubian monoidal category $\mathcal{C}$. 
 Then a selfdual idempotent $b \in \End_{\mathcal{C}}(X)$ arises from a Frobenius subalgebra if and only if it satisfies the exchange relations.
 \end{cor}
 
 \begin{proof}
 By Theorem~\ref{thm:main}, it remains to show that the idempotent $b \in \End_{\mathcal{C}}(X)$ is always unital. 
 If $b$ satisfies the exchange relations, this follows from the last statement of Lemma~\ref{ERidemlem}. 
 If $b$ comes from a Frobenius subalgebra, it holds by Definition~\ref{Frobsubalgdefn}.
 %
 %
 \end{proof}
 
 Recall the notion of the Fourier transform from Definition~\ref{def:fourier}. 
 
 \begin{cor} \label{cor:cocentral}
 Let $X$ be a Frobenius algebra in a Karoubian monoidal category $\mathcal{C}$. 
 Every unital selfdual idempotent whose Fourier transform is central arises from a Frobenius subalgebra.
 \end{cor}
 \begin{proof}
 This follows directly from Proposition~\ref{normalERprop} and Theorem~\ref{thm:main}.
 \end{proof}
 
 Next, recall the notion of a normal biprojection introduced in Definitions~\ref{def:bipro} and~\ref{def:normalbipro}.
 
 \begin{cor} \label{cor:normalbi}
 Let $X$ be a Frobenius algebra in a linear Karoubian monoidal category $\mathcal{C}$. 
 Then every normal biprojection $b \in \End_{\mathcal{C}}(X)$ arises from a Frobenius subalgebra.
 \end{cor}
 \begin{proof}
 This is an immediate consequence of Corollary~\ref{cor:cocentral}.
 \end{proof}
 
\comments{We now present a connected unitary analogue of Theorem~\ref{thm:main} (see Remark~\ref{rk:DualVsAdjoint}).}
 
A classical result in subfactor theory~\cite{B94,BJ00} states that if $N \subseteq M$ is an irreducible, finite-index inclusion of type $\mathrm{II}_1$ factors, then a unital self-adjoint idempotent in $N' \cap M_1$ arises from an intermediate subfactor if and only if it satisfies the exchange relations. 
Theorem~\ref{thm:uniMain}, which is the connected unitary analogue of Theorem~\ref{thm:main}, can be deduced from this classical result by noting that every unitary Frobenius algebra is realized as the standard invariant of some irreducible, finite-index subfactor. We provide instead a purely categorical and self-contained proof.

 \begin{thm}\label{thm:uniMain}
 Let $X$ be a connected unitary Frobenius algebra in a unitary tensor category $\mathcal{C}$. 
 A unital self-adjoint idempotent $b:X\to X$ arises from a unitary Frobenius subalgebra if and only if it satisfies the exchange relations.
 \end{thm}
 
 \begin{proof}
 Suppose $b$ arises from a unitary Frobenius subalgebra $Y \subseteq X$, i.e., there exists a unital algebra monomorphism $i: Y \to X$ such that $i^\dagger \circ i = \id_Y$ (an isometry) and $b = i \circ i^\dagger$ (cf.~the proof of~\cite[Lemma~4.28]{GP25}). 
 Recall from~\cite[Definition~4.25]{GP25} that a unitary Frobenius algebra satisfies $m^\dagger = \delta$ and $e^\dagger = \epsilon$. 
 Applying the adjoint $\dagger$ to the algebra morphism identity $m \circ (i \otimes i) = i \circ m$ shows that $i^\dagger$ is a coalgebra morphism. 
 The diagram below shows that $b$ is convolution-stable:
 %
 \vspace*{-4mm}
 \[b * b = \raisebox{-11mm}{
 \begin{tikzpicture}
 \draw[blue,in=90,out=90,looseness=2] (0,0) to (1,0);
 \draw[blue,in=-90,out=-90,looseness=2] (0,0) to (1,0);
 \draw[blue] (.5,-.6) to (.5,-1.2);
 \draw[blue] (.5,.6) to (.5,1.2);
 \node[draw,thick,rounded corners, fill=white,minimum width=20,scale=.8] at (0,0){$b$};
 \node[draw,thick,rounded corners, fill=white,minimum width=20,scale =.8] at (1,0){$b$}; 
 \end{tikzpicture}}
 =
 \raisebox{-11mm}{
 \begin{tikzpicture}
 \draw[blue,in=90,out=90,looseness=2] (0,0) to (1,0);
 \draw[blue,in=-90,out=-90,looseness=2] (0,0) to (1,0);
 \draw[blue] (.5,-.6) to (.5,-1.2);
 \draw[blue] (.5,.6) to (.5,1.2);
 \node[draw,thick,rounded corners, fill=white,minimum width=20,scale=.7] at (0,.2){$i^\dagger$};
 \node[draw,thick,rounded corners, fill=white,minimum width=20,scale=.7] at (1,.2){$i^\dagger$};
 \node[draw,thick,rounded corners, fill=white,minimum width=20,scale=.7] at (0,-.3){$i$};
 \node[draw,thick,rounded corners, fill=white,minimum width=20,scale=.7] at (1,-.3){$i$}; 
 \end{tikzpicture}}
 = 
 \raisebox{-11mm}{
 \begin{tikzpicture}
 \draw[blue,in=90,out=90,looseness=2] (0,0) to (1,0);
 \draw[blue,in=-90,out=-90,looseness=2] (0,0) to (1,0);
 \draw[blue] (.5,-.6) to (.5,-1.2);
 \draw[blue] (.5,.6) to (.5,1.2);
 \node[draw,thick,rounded corners, fill=white,minimum width=20,scale=.7] at (.5,.9){$i^\dagger$};
 \node[draw,thick,rounded corners, fill=white,minimum width=20,scale=.7] at (.5,-.9){$i$}; 
 \end{tikzpicture}} 
 = \lambda_Y b, 
 \vspace*{-1mm}
 \]
 where $\lambda_Y$ arises from the separability of $Y$ (since $Y$ is connected, like $X$, and hence separable by \cite[Lemma~2.16]{GP25}). 
 The third equality holds because $i$ (respectively $i^\dagger$) is an algebra (respectively coalgebra) morphism. 
 The exchange relations assert that $A = B = 0$, where
 \vspace*{-1mm}
 \[A:=
 \raisebox{-10mm}{
 \begin{tikzpicture}
 \draw[blue,in=90,out=90,looseness=2] (-0.5,0.5) to (-1.5,0.5);
 \draw[blue] (-1,1.1) to (-1,2);
 \draw[blue] (-.5,.5) to (-.5,-.2);
 \draw[blue] (-1.5,.5) to (-1.5,-.2);
 \node[draw,thick,rounded corners, fill=white,minimum width=15,scale=.8] at (-1.5,.4){$b$};
 \node[draw,thick,rounded corners, fill=white,minimum width=15,scale=.8] at (-1,1.6){$b$};
 \end{tikzpicture}}
 -
 \raisebox{-10mm}{
 \begin{tikzpicture}
 \draw[blue,in=90,out=90,looseness=2] (-0.5,0.5) to (-1.5,0.5);
 \draw[blue] (-1,1.1) to (-1,2);
 \draw[blue] (-.5,.5) to (-.5,-.2);
 \draw[blue] (-1.5,.5) to (-1.5,-.2);
 \node[draw,thick,rounded corners, fill=white,minimum width=15,scale=.8] at (-1.5,.4){$b$};
 \node[draw,thick,rounded corners, fill=white,minimum width=15,scale=.8] at (-.5,.4){$b$};
 \end{tikzpicture}} 
 \text{ and }
 B:=
 \raisebox{-10mm}{
 \begin{tikzpicture}
 \draw[blue,in=90,out=90,looseness=2] (-0.5,0.5) to (-1.5,0.5);
 \draw[blue] (-1,1.1) to (-1,2);
 \draw[blue] (-.5,.5) to (-.5,-.2);
 \draw[blue] (-1.5,.5) to (-1.5,-.2);
 \node[draw,thick,rounded corners, fill=white,minimum width=15,scale=.8] at (-.5,.4){$b$};
 \node[draw,thick,rounded corners, fill=white,minimum width=15,scale=.8] at (-1,1.6){$b$};
 \end{tikzpicture}}
 -
 \raisebox{-10mm}{
 \begin{tikzpicture}
 \draw[blue,in=90,out=90,looseness=2] (-0.5,0.5) to (-1.5,0.5);
 \draw[blue] (-1,1.1) to (-1,2);
 \draw[blue] (-.5,.5) to (-.5,-.2);
 \draw[blue] (-1.5,.5) to (-1.5,-.2);
 \node[draw,thick,rounded corners, fill=white,minimum width=15,scale=.8] at (-1.5,.4){$b$};
 \node[draw,thick,rounded corners, fill=white,minimum width=15,scale=.8] at (-.5,.4){$b$};
 \end{tikzpicture}} 
 \vspace*{-1mm}
 \]
 By unitarity, it suffices to show $A^\dagger A = B^\dagger B = 0$. 
 Since $b$ is convolution-stable idempotent,
 \vspace*{-1mm}
 \[
 A^{\dagger}A = 
 \raisebox{-11mm}{
 \begin{tikzpicture}
 \draw[blue,in=90,out=90,looseness=2] (0,0) to (1,0);
 \draw[blue,in=-90,out=-90,looseness=2] (0,0) to (1,0);
 \draw[blue] (.5,-.6) to (.5,-1.2);
 \draw[blue] (.5,.6) to (.5,1.2);
 \node[draw,thick,rounded corners, fill=white,minimum width=20,scale=.7] at (0,0){$b$};
 \node[draw,thick,rounded corners, fill=white,minimum width=20,scale=.7] at (.5,.9){$b$};
 \node[draw,thick,rounded corners, fill=white,minimum width=20,scale=.7] at (.5,-.9){$b$};
 \end{tikzpicture}}
 -
 \raisebox{-11mm}{
 \begin{tikzpicture}
 \draw[blue,in=90,out=90,looseness=2] (0,0) to (1,0);
 \draw[blue,in=-90,out=-90,looseness=2] (0,0) to (1,0);
 \draw[blue] (.5,-.6) to (.5,-1.2);
 \draw[blue] (.5,.6) to (.5,1.2);
 \node[draw,thick,rounded corners, fill=white,minimum width=20,scale=.7] at (0,0){$b$};
 \node[draw,thick,rounded corners, fill=white,minimum width=20,scale =.7] at (1,0){$b$}; 
 \node[draw,thick,rounded corners, fill=white,minimum width=20,scale=.7] at (.5,-.9){$b$};
 \end{tikzpicture}}
 -
 \raisebox{-11mm}{
 \begin{tikzpicture}
 \draw[blue,in=90,out=90,looseness=2] (0,0) to (1,0);
 \draw[blue,in=-90,out=-90,looseness=2] (0,0) to (1,0);
 \draw[blue] (.5,-.6) to (.5,-1.2);
 \draw[blue] (.5,.6) to (.5,1.2);
 \node[draw,thick,rounded corners, fill=white,minimum width=20,scale=.7] at (0,0){$b$};
 \node[draw,thick,rounded corners, fill=white,minimum width=20,scale =.7] at (1,0){$b$}; 
 \node[draw,thick,rounded corners, fill=white,minimum width=20,scale=.7] at (.5,.9){$b$}; 
 \end{tikzpicture}}
 +
 \raisebox{-11mm}{
 \begin{tikzpicture}
 \draw[blue,in=90,out=90,looseness=2] (0,0) to (1,0);
 \draw[blue,in=-90,out=-90,looseness=2] (0,0) to (1,0);
 \draw[blue] (.5,-.6) to (.5,-1.2);
 \draw[blue] (.5,.6) to (.5,1.2);
 \node[draw,thick,rounded corners, fill=white,minimum width=20,scale=.7] at (0,0){$b$};
 \node[draw,thick,rounded corners, fill=white,minimum width=20,scale =.7] at (1,0){$b$}; 
 \end{tikzpicture}} 
 = 
 \raisebox{-11mm}{
 \begin{tikzpicture}
 \draw[blue,in=90,out=90,looseness=2] (0,0) to (1,0);
 \draw[blue,in=-90,out=-90,looseness=2] (0,0) to (1,0);
 \draw[blue] (.5,-.6) to (.5,-1.2);
 \draw[blue] (.5,.6) to (.5,1.2);
 \node[draw,thick,rounded corners, fill=white,minimum width=20,scale=.7] at (0,0){$b$};
 \node[draw,thick,rounded corners, fill=white,minimum width=20,scale=.7] at (.5,.9){$b$};
 \node[draw,thick,rounded corners, fill=white,minimum width=20,scale=.7] at (.5,-.9){$b$};
 \end{tikzpicture}}
 - \lambda_Y b.
 \vspace*{-1mm} 
 \]
 By connectedness and \cite[Lemma~2.24]{GP25},
 \(
 \mu_X A^{\dagger}A = (\trace(b) - \mu_X \lambda_Y)\, b,
 \)
 where $\mu_X := \epsilon_X \circ e_X$. By \cite[Lemma~2.19]{GP25}, 
 \(
 \trace(b) = \trace(i \circ i^{\dagger}) = \trace((i^{\dagger})^{**} \circ i).
 \)
 Since $i^{\dagger} \circ i = \id_Y$ and $(i^{\dagger})^* = (i^*)^{\dagger}$ by \cite[Lemma~4.28]{GP25},
 \(
 \trace(b) = \trace(\id_Y) = \mu_Y \lambda_Y
 \)
 by \cite[Lemma~2.23]{GP25}. Next, 
 \[
 \mu_Y = \epsilon_Y \circ e_Y 
 = \epsilon_Y \circ \id_Y \circ e_Y 
 = \epsilon_Y \circ i^{\dagger} \circ i \circ e_Y 
 = \epsilon_X \circ e_X 
 = \mu_X,
 \]
 since $i$ is unital (and hence $i^{\dagger}$ is counital). Therefore, 
 \(
 \mu_X A^{\dagger}A = 0.
 \)
 But $\mu_X = e_X^{\dagger} \circ e_X$ is nonzero, hence $A^{\dagger}A = 0$. 
 Similarly, $B^{\dagger}B = 0$ since $\trace(b^*) = \trace(b)$ as $Y^* = Y$.
 
 
 Conversely, assume that $b$ is a self-adjoint unital idempotent in $\mcal C$ satisfying the exchange relations. 
 By applying the unitary idempotent completion~\cite{CPJP22} to $b$ in the unitary tensor category $\mcal C$, we get an object $Y$ and an isometry $i : Y \to X$ such that $b = i \circ i^\dagger$ and $i^\dagger \circ i = \id_Y$.
 The rest of the proof follows verbatim from the converse part of Theorem~\ref{thm:main}, except that self-adjointness and the graphical calculus for unitary tensor categories are used in place of selfduality and the graphical calculus for Karoubian monoidal categories, respectively.
 \end{proof}
 \begin{question} \label{Qu:UnitarySub}
 Can the connectedness assumption in Theorem~\ref{thm:uniMain} be omitted?
 \end{question}
 
 \begin{remark} \label{rk:DualVsAdjoint}
 The connectedness condition is unnecessary for the backward implication. It remains unclear whether the forward implication holds without it---in agreement with the last sentence of \cite[Remark~3.4(2)]{B94}. Surprisingly, Theorem~\ref{thm:main} does not require connectedness in its general setting with selfdual idempotents, whereas Theorem~\ref{thm:uniMain} does, in its unitary setting with self-adjoint idempotents. If Question~\ref{Qu:UnitarySub2} has a positive answer, we could replace \emph{self-adjoint} by \emph{selfdual} in Theorem~\ref{thm:uniMain} and omit the connectedness assumption.
 \end{remark}
 
 \begin{question} \label{Qu:UnitarySub2}
 Is every Frobenius subalgebra of a unitary Frobenius algebra itself unitary?
 \end{question}
 
 \comments{ As a first application of Theorem~\ref{thm:uniMain}, we recover \cite[Proposition~3.6]{BJ00}.
 
 \begin{cor} \label{cor:intervN}
 Let $N \subseteq M$ be an irreducible finite-index inclusion of $\rm{II}_1$ factors. 
 A unital self-adjoint idempotent in $N' \cap M_1$ arises from an intermediate subfactor if and only if it satisfies the exchange relations.
 \end{cor}}
 
 We recall the notion of an averaging operator on a Banach algebra:
 
 \begin{defn}[\cite{M66}] \label{def:average}
 Let $\mathcal{A}$ be a Banach algebra over the complex field, and let $\mathcal{B}(\mathcal{A})$ denote the Banach algebra of all bounded linear operators on $\mathcal{A}$. An operator $T \in \mathcal{B}(\mathcal{A})$ satisfying the exchange relations, i.e. $T(T(a)b) = T(a)T(b) = T(aT(b))$, for all $a$, $b \in \mathcal{A}$, is called an \emph{averaging operator}.
 \end{defn}
 
 \begin{cor}\label{cor:C*algthm}
 Let $(A \subseteq B, E)$ be an irreducible finite-index unital inclusion of C*-algebras. 
 Then an $A$-bimodule unital self-adjoint morphism $T : {_A}B_A \to {_A}B_A$ arises from an intermediate C*-subalgebra if and only if it is an averaging operator.
 \end{cor}
 
 \begin{proof}
 Consider the unitary tensor category $\mathrm{Bim}(A)$, whose objects are $A$-bimodules and morphisms are adjointable. 
 Applying Theorem~\ref{thm:uniMain} to the unitary Frobenius algebra $_A B_A$ in $\mathrm{Bim}(A)$ yields the desired result, since every intermediate C*-algebra corresponds to a unitary Frobenius subalgebra of $_A B_A$ (via the realization functor described in \cite{CPJP22}).
 \end{proof}
 \
 %
 %
 %
 
 \section{Appendix} \label{sec:appendix}
 
 \subsection{SageMath code and computations} \label{sub:sage}
 This section provides a computer-assisted proof of Propositions~\ref{prop:bipro} and~\ref{prop:bipro2}. 
 Using \texttt{SageMath}, we verify that in Lemma~\ref{lem:12345}, the equations~\eqref{lem_5} lie in the ideal generated by equations~\eqref{lem_1}--\eqref{lem_4}. To this end, we use the function \texttt{make\_eqs} defined below.
 {\footnotesize
 \begin{verbatim}
 from itertools import product
 def make_eqs(n):
 names = [f'A_{i}_{j}_{k}_{l}' for i,j,k,l in product(range(1,n+1),repeat=4)] + ['x']
 R = PolynomialRing(QQ, names)
 gens = R.gens()
 A = { p: gens[i] for i,p in enumerate(product(range(1,n+1),repeat=4)) }
 x = gens[-1]
 eqs1 = [ A[i,j,k,l] - A[l,k,j,i] for i,j,k,l in product(range(1,n+1),repeat=4) ]
 eqs2 = [ sum(A[i,i,k,l] for i in range(1,n+1)) - (1 if k==l else 0) 
 for k,l in product(range(1,n+1),repeat=2) ]
 eqs3 = [ sum(A[i,j,r,s]*A[r,s,k,l] for r,s in product(range(1,n+1),repeat=2))
 - A[i,j,k,l] for i,j,k,l in product(range(1,n+1),repeat=4) ]
 eqs4 = [ sum(A[i,r,k,s]*A[r,j,s,l] for r,s in product(range(1,n+1),repeat=2))
 - x*A[i,j,k,l] for i,j,k,l in product(range(1,n+1),repeat=4) ]
 eqs5 = []
 for i,j,k,l,r,s in product(range(1,n+1),repeat=6):
 t1 = sum(A[i,j,t,k]*A[t,l,r,s] for t in range(1,n+1))
 t2 = sum(A[i,j,r,t]*A[k,l,t,s] for t in range(1,n+1))
 t3 = sum(A[k,l,j,t]*A[i,t,r,s] for t in range(1,n+1))
 eqs5.append(t1 - t2); eqs5.append(t2 - t3)
 return R, eqs1, eqs2, eqs3, eqs4, eqs5
 
 \end{verbatim}}
 We now present the proof of Proposition~\ref{prop:bipro} in the case $n = 2$:
 {\small
 \begin{verbatim}
 sage: R, eqs1, eqs2, eqs3, eqs4, eqs5 = make_eqs(2) # for n=2
 sage: I = R.ideal(eqs1 + eqs2 + eqs3 + eqs4)
 sage: all(eq in I for eq in eqs5)
 True 
 \end{verbatim}}
 In fact, the algebraic set of biprojections is $2$-dimensional, whereas the one of selfdual unital ER idempotents is $3$-dimensional:
 {\small
 \begin{verbatim}
 sage: J = R.ideal(eqs1 + eqs2 + eqs3 + eqs5)
 sage: [I.dimension(), J.dimension()]
 [2,3]
 \end{verbatim}}
 We now present the proof of Proposition~\ref{prop:bipro2} in the case $n = 2$:
 {\small
 \begin{verbatim}
 sage: (x-1/2)*(x-1)*(x-2) in I
 True
 sage: Ihalf = R.ideal(eqs1 + eqs2 + eqs3 + eqs4 + [x-1/2])
 sage: Ihalf.vector_space_dimension()
 1 # only one solution
 sage: I1 = R.ideal(eqs1 + eqs2 + eqs3 + eqs4 + [x-1])
 sage: I1.dimension()
 2 # 2-dim set of solutions
 sage: I2 = R.ideal(eqs1 + eqs2 + eqs3 + eqs4 + [x-2])
 sage: I2.vector_space_dimension()
 1 # only one solution
 \end{verbatim}}
 The computation for $n = 3$ proceeds similarly.
 \begin{ac} 
 The first author is supported by the BJNSF (Grant No. 1S24063). The second author is supported by the NSFC (Grant No.12471031).
 \end{ac}


\begin{thebibliography}{1}
 \bibitem{A99}
 {\sc L.~Abrams}, {\em Modules, comodules and cotensor products over Frobenius algebras}, J. Alg. 219 (1999), 201–213.
 
 \bibitem{B94}
 {\sc D.~Bisch}, {\em A note on intermediate subfactors}, Pac. J. Math. 163 (1994), no. 2, 201–216.
 
 \bibitem{BJ00}
 {\sc D.~Bisch, V.F.R.~Jones}, {\em Singly generated planar algebras of small dimension}, Duke Math. J. 101 (2000), no. 1, 41–75.
 
 \bibitem{CPJP22}
 {\sc Q.~Chen, R.~Hernandez Palomares, C.~Jones, D.~Penneys}, {\em Q-system completion for C*-2-categories}, J. Funct. Anal. 283 (2022), no. 3, Paper No. 109524, 59 pp.
 
 \bibitem{FS08}
 {\sc J.~Fuchs, C.~Stigner}, {\em On Frobenius algebras in rigid monoidal categories}, Arab. J. Sci. Eng., Sect. C Theme Issues 33 (2008), no. 2, 175–191.
 
 \bibitem{GM68}
 {\sc J.L.B.~Gamlen, J.B.~Miller}, {\em Averaging and Reynolds Operators on Banach Algebras II: Spectral Properties of Averaging Operators}, J. Math. Anal. Appl. 23 (1968), 183–197.
 
 \bibitem{GP25}
 {\sc M.~Ghosh, S.~Palcoux}, {\em Frobenius subalgebra lattices in tensor categories}, arXiv:2502.19876v9.
 
 \bibitem{J83}
 {\sc V.F.R.~Jones}, {\em Index for subfactors}, Invent. Math. 73 (1983), 1–25.
 
 \bibitem{L02}
 {\sc Z.~Landau}, {\em Exchange relation planar algebras}, in: Proceedings of the Conference on Geometric and Combinatorial Group Theory, Part II (Haifa, 2000), Geom. Dedicata 95 (2002), 183–214.
 
 \bibitem{L16}
 {\sc Z.~Liu}, {\em Exchange relation planar algebras of small rank}, Trans. Amer. Math. Soc. 368 (2016), no. 12, 8303–8348.
 
 \bibitem{L94}
 {\sc R.~Longo}, {\em A duality for Hopf algebras and for subfactors I}, Comm. Math. Phys. 159 (1994), 133–150.
 
 \bibitem{M66}
 {\sc J.B.~Miller}, {\em Averaging and Reynolds operators on Banach algebras I: Representation by derivations and antiderivations}, J. Math. Anal. Appl. 14 (1966), 527–548.
 
 \bibitem{M03}
 {\sc M.~Müger}, {\em From subfactors to categories and topology I: Frobenius algebras and Morita equivalence of tensor categories}, J. Pure Appl. Algebra 180 (2003), 81–157.
 
 \bibitem{P18}
 {\sc S.~Palcoux}, {\em Ore's theorem for cyclic subfactor planar algebras and beyond}, Pac. J. Math. 292 (2018), no. 1, 203–221.
 
 \bibitem{P95}
 {\sc S.~Popa}, {\em An axiomatization of the lattice of higher relative commutants}, Invent. Math. 120 (1995), 237–252.
 
 \bibitem{Q95}
 {\sc F.~Quinn}, {\em Lectures on axiomatic topological quantum field theory}, in: Geometry and Quantum Field Theory (Park City, UT, 1991), IAS/Park City Math. Ser., 1, Amer. Math. Soc., 1995, 323–453.
 
 \bibitem{SY11}
 {\sc A.~Skowroński, K.~Yamagata}, {\em Frobenius algebras I: Basic representation theory}, EMS Textbooks in Mathematics, European Mathematical Society (EMS), Zürich, 2011, xii+650 pp.
 
 \bibitem{T98}
 {\sc T.~Teruya}, {\em Normal intermediate subfactors}, J. Math. Soc. Japan 50 (1998), no. 2, 469–490.
 
 \end{thebibliography}
 \end{document}